\documentclass[11pt]{amsart}
\usepackage{amssymb, amsmath, amsthm, amsfonts, enumerate}
\usepackage{amsmath,setspace,scalefnt}
\usepackage[dvipsnames,table]{xcolor}
\usepackage{graphicx,tikz,caption,subcaption}
\usepackage{amsmath,amssymb}
\usepackage{graphicx}
\usepackage{theoremref}
\usepackage[hidelinks]{hyperref}
\usepackage{comment}
\usepackage[ruled]{algorithm2e}
\usepackage{fullpage}
\usepackage{enumerate}
\usepackage{pgf,tikz}
\usepackage{todonotes}
\usepackage{thm-restate}
\usetikzlibrary{arrows}
\usetikzlibrary{positioning}
\usepackage[margin=1in]{geometry}
\usepackage[skip=10pt plus2pt, indent=15pt]{parskip}
\usepackage{array}
\usepackage[foot]{amsaddr}

\def\COMMENT#1{}
\def\TASK#1{}

\newtheorem{thm}{Theorem}[section]

\newtheorem{lem}[thm]{Lemma}

\newtheorem{prop}[thm]{Proposition}

\newtheorem{ques}[thm]{Question}
\newtheorem{prob}[thm]{Problem}

\newtheorem*{thm*}{Theorem}
\newtheorem*{define*}{Definition}
\newtheorem*{examp*}{Example}
\newtheorem*{lem*}{Lemma}
\newtheorem*{claim*}{Claim}
\newtheorem*{fact*}{Fact}
\newtheorem*{col*}{Corollary}
\newtheorem*{conj*}{Conjecture}

\newcommand{\Ex}{\mathbb{E}}

\renewcommand{\Pr}{\mathbb{P}}

\newcommand{\cC}{\mathcal{C}}
\newcommand{\bG}{{\bf G}}
\newcommand{\ep}{\varepsilon}

\usepackage{centernot}

\tikzstyle{std}=[ circle, draw=black,fill=black, inner sep=0pt, minimum size=2mm]
\tikzstyle{stdsm}=[ circle, draw=black,fill=black, inner sep=0pt, minimum size=1.5mm]
\tikzstyle{set}=[ circle, draw=black,fill=none, inner sep=0pt, minimum size=15mm]
\tikzstyle{stdr}=[ circle, draw=black,fill=red, inner sep=0pt, minimum size=2mm]
\tikzstyle{stdb}=[ circle, draw=black,fill=cyan, inner sep=0pt, minimum size=2mm]

\makeatletter
\newcommand{\xnrightarrow}[2][]{%
  \mathrel{%
    \vphantom{\xrightarrow[#1]{#2}}%
    \ooalign{\hidewidth\neg@arrow\hidewidth\cr$\m@th\xrightarrow[#1]{#2}$\cr}%
  }%
}
\newcommand{\neg@arrow}{%
  $\m@th\vcenter{\hbox{%
    \rotatebox[origin=c]{-45}{\scalebox{1.5}[1]{$\m@th\scriptscriptstyle|$}}%
  }}$
}
\makeatother

\DeclareTextCompositeCommand{\v}{OT1}{l}{l\nobreak\hspace{-.1em}'}
\DeclareTextCompositeCommand{\v}{OT1}{t}{t\nobreak\hspace{-.1em}'\nobreak\hspace{-.15em}}

\begin{document}

\title[]{Universality for Transversal Powers of Hamilton Cycles}
\author{Emily Heath$^{\S}$, Joseph Hyde$^{\parallel \nabla}$, Natasha Morrison$^{\ast \dagger}$, \and Shannon Ogden$^{\ast \ddagger}$}
\address{\tiny$^{\S}$Department of Mathematics and Statistics, California State Polytechnic University Pomona. Email: \href{mailto:eheath@cpp.edu}{\tt eheath@cpp.edu}.}
\address{\tiny$^{\parallel}$ Department of Mathematics, Strand Building, Strand Campus, Strand, London, WC2R 2LS. Email: \href{mailto:joseph.hyde@kcl.ac.uk}{\tt joseph.hyde@kcl.ac.uk}.}
\address{\tiny$^{\nabla}$ Supported by UK Research and Innovation Future Leaders Fellowship MR/W007320/2.}
\address{\tiny$^{\ast}$Department of Mathematics and Statistics, University of Victoria.}
\address{\tiny$^{\dagger}$Research supported by NSERC Discovery Grant RGPIN-2021-02511. Email: \href{mailto:nmorrison@uvic.ca}{\tt nmorrison@uvic.ca}.}
\address{\tiny$^{\ddagger}$Supported by Vanier Canada Graduate Scholarship. Email: \href{mailto:sogden@uvic.ca}{\tt sogden@uvic.ca}.}

\maketitle
\begin{abstract}
    
    Let $k \ge 2$ and let $\bG = \{G_1, \ldots, G_{m}\}$ be a collection of graphs on a common vertex set of cardinality $n$. We show that if each graph in $\bG$ has minimum degree at least $(1-\frac{1}{2k} + o(1))n$, then for every edge-colouring $\chi$ of the $k$th power of a Hamilton cycle $C_n^k$ with $m$ colours, there is a copy of $C_n^k$ in $\bG$ such that $e \in G_{\chi(e)}$ for every edge $e$ in $C_n^k$. This generalises a result of Bowtell, Morris, Pehova, and Staden, who provided asymptotically best possible minimum degree conditions for the Hamilton cycle.
    
\end{abstract}

\section{Introduction}
\label{Sec:Intro}

Given two graphs $G$ and $H$, a natural and well-studied question in graph theory is whether $H$ is a subgraph of $G$, that is, if $H$ can be embedded in $G$.
For example, if $G$ has order $n\ge 3$ and $H=C_n$, then this is the problem of determining whether $G$ contains a Hamilton cycle, which is known to be NP-complete~\cite{K}.
As a result, there has been an extensive search for simple sufficient conditions that guarantee the existence of a Hamilton cycle. The most famous result in this area is due to Dirac~\cite{D},
which states that any graph of order $n\ge 3$ with minimum degree at least $\frac{n}{2}$ will contain a Hamilton cycle. 

In 2020, Joos and Kim~\cite{JK} introduced a natural generalization of the general graph embedding problem to graph collections.
Let $\bG=\{G_1,\ldots,G_m\}$ be a collection of graphs on a common vertex set $V$. Note that here, and throughout this paper, the graphs in $\bG$ need not be distinct. 
Now, given a graph $H$ with $|E(H)|\le m$,
a \emph{transversal} or \emph{rainbow} copy of $H$ in $\bG$ is a copy of $H$ for which there exists an injective function $\chi:E(H)\rightarrow [m]$ such that $e\in G_{\chi(e)}$ for all $e\in E(H)$.
The \emph{transversal embedding problem} is then to determine which conditions on the graphs in $\bG$ guarantee the existence of a rainbow copy of $H$. 

Suppose that any graph $G$ satisfying condition $\cC$ has $H$ as a subgraph. 
Of particular interest is the question of whether requiring each $G_i\in \bG$ to satisfy $\cC$ is enough to guarantee the existence of a rainbow copy of $H$ in $\bG$, 
in which case we refer to $H$ as being \emph{colour-blind} with respect to $\bG$.
Note that, if all graphs in $\bG$ are identical and satisfy $\cC$,
then this is simply the original embedding problem. 
Answering a question of Aharoni (see~\cite{ADHMS}), Joos and Kim~\cite{JK} proved that the transversal generalization of Dirac's Theorem is colour-blind; that is,
any collection $\bG=\{G_1,\ldots,G_n\}$ of graphs on a common vertex set $V$ of order $n$ such that $\delta(G_i)\ge \frac{n}{2}$ for all $i\in[n]$ will contain a transversal Hamilton cycle. Transversal results have been obtained for many families of graphs, see for example \cite{AC,CHWW,CIKL,CKLS,GHMPS,JK,MMP} and the survey of Sun, Wang and Wei~\cite{SWW}.

A natural generalisation of the transversal problem is to require the existence of a copy of $H$ with a specified colour pattern.
For $\bG=\{G_1,\ldots,G_m\}$ and a fixed edge-colouring $\chi:E(H)\rightarrow [m]$, say an edge-coloured graph $H$ in $\bG$ is $\chi$\emph{-coloured} if, 
for every $e\in E(H)$, we have $e\in G_{\chi(e)}$. 
We refer to $\chi$ as a \emph{colour pattern} of $H$, and view each $G_i\in \bG$ as being in colour $i$. Note that, when $m=1$ (or equivalently when all graphs in $\bG$ are identical),
the existence of a copy of $H$ with a specified colour pattern is simply the standard graph embedding problem. 

\begin{prob}
    \thlabel{Problem}
    Given a collection $\bG=\{G_1,\ldots,G_m\}$ of graphs on a common vertex set $V$,
    and a graph $H$, what conditions on $\bG$ will guarantee that, for any edge-colouring $\chi:E(H)\rightarrow[m]$,
    there exists a $\chi$-coloured copy of $H$ in $\bG$?
\end{prob}

In comparison to the transversal embedding problem, \thref{Problem} has been much less widely studied to date, and the main focus of study has concerned minimum degree conditions. To the best of our knowledge, the first results in this area were proved by Montgomery, M\"{u}yesser and Pehova~\cite{MMP}. They found minimum degree conditions on $\bG$ which suffice to find factors, including an asymptotically best possible result for clique factors (via a result of Keevash and Mycroft~\cite{KM}). Given $\bG=\{G_1,\ldots,G_m\}$, let $\delta(\bG):=\min\{\delta(G_i):i\in [m]\}$. Bowtell, Morris, Pehova, and Staden~\cite{BMPS} 
showed that $\delta(\bG)\ge \left(\frac{1}{2}+\alpha\right)n$ is sufficient to guarantee the existence of any colour pattern of a Hamilton cycle, thus providing an asymptotically best possible answer to \thref{Problem} in this case.  This result has recently been extended by Christoph, Martinsson and Milojevi\'{c}~\cite{CMM} to the setting
of random graphs, answering a question of Pehova. 

Our main result builds upon and generalises the result of~\cite{BMPS} to powers of a Hamilton cycle. The \emph{$k$th power of a graph} $G$, 
denoted by $G^k$, is the graph obtained from $G$ by adding edges between all vertices $x$ and $y$ such that $2\le d(x,y)\le k$.

\begin{restatable}{thm}{MainTheoremGen}
\thlabel{thm:main Gen}
    Let $k\ge 2$. For every $\alpha>0$ there exists $n_0 := n_0(\alpha,k)$ such that for every $n \geq n_0$ the following holds. 
    Let $\bG = \{G_1, \ldots, G_{m}\}$ be a collection of graphs on common vertex set $V$ with $|V| = n$
    such that 
    $\delta(\bG) \geq (1-\frac{1}{2k} + \alpha)n$. Then, for every edge-colouring $\chi:E(C_n^k)\rightarrow [m]$, there exists a $\chi$-coloured $k$th power of a Hamilton cycle in $\bG$. 
\end{restatable}

The result of Bowtell, Morris, Pehova, and Staden~\cite{BMPS} shows that asymptotically the degree condition required for Dirac's theorem will suffice to find all colour patterns of a Hamilton cycle. In 1962, P\'{o}sa conjectured an analogue of Dirac's Theorem for squares of Hamilton cycles, 
which was later further generalized by Seymour~\cite{S} in 1974 to any power of a Hamilton cycle. 
Koml\'{o}s, S\'{a}rk\"{o}zy, and Szemer\'{e}di~\cite{KSS} confirmed this conjecture in 1998 for large graphs: For $n$ sufficiently large, if $G$ is a graph of order $n$ with $\delta(G)\ge (1-\frac{1}{k+1})n$, then $G$ contains the $k$th power of a Hamilton cycle. Interestingly, our proof of \thref{thm:main Gen} requires $\delta(\bG)\ge (1-\frac{1}{2k}+\alpha)n$ in multiple places. We discuss this discrepancy in Section~\ref{Sec:Future}.

The rest of the paper is organized as follows:
In Section~\ref{Sec:Overview} we present a general overview of the proof of \thref{thm:main Gen}, as well as the statements of the two main lemmas (\thref{lem:absorb,lem:kpathcollection}) that will be required. 
Section~\ref{Sec:Prelim} provides the notation that will be used throughout the paper,
along with a few basic results which will be useful in later sections. 
The proof of \thref{thm:main Gen} can be found in Section~\ref{Sec:Main}, while the proofs of \thref{lem:absorb,lem:kpathcollection} can be found in Sections \ref{Sec:Absorb} and \ref{Sec:Cover}, respectively. 
We conclude in Section \ref{Sec:Future} with a discussion of possible directions for future work.

\section{Main lemmas and proof overview}
\label{Sec:Overview}

In this section, we provide a general overview of our strategy for proving \thref{thm:main Gen}, and present the two main lemmas we will use, namely \thref{lem:absorb,lem:kpathcollection}. 
Note that, for the sake of brevity, we will refer to the $k$th power of a cycle as a $k$\emph{-cycle}, and similarly refer to the $k$th power of a path as a $k$\emph{-path}.

The proof of \thref{thm:main Gen} will follow a similar absorption strategy to that of Bowtell, Morris, Pehova, and Staden in~\cite{BMPS}.
Broadly speaking, we will be constructing many $k$-paths of constant length via a simple random process,
greedily connecting these paths together to form one long $k$-path, before finally completing the Hamilton $k$-cycle by using an absorbing structure,
which had been set aside at the start, to use up the remaining vertices.

More specifically, we will begin by partitioning the vertex set $V$ of the graph collection $\bG$ into two sets $Z$ and $V\setminus Z$, so that every vertex in $V$ has sufficiently high degree into both $Z$ and $V\setminus Z$, which is made possible by \thref{prop:zexistence}.
We will view this set $Z$ as the reservoir through which we will seek to connect the various $k$-paths we construct along the way. 

Having partitioned the vertex set, we will then set aside an absorbing set $A\subseteq V\setminus Z$ on which we will build our absorbing structure. 
This absorbing structure will, when given any small subset $Z'\subseteq Z$, produce a $k$-path with a fixed colour pattern that uses exactly the vertices in $Z'\cup A$. As noted by Bowtell, Morris, Pehova, and Staden in~\cite{BMPS}, this absorbing structure must only allow for flexibility in terms of the vertices chosen, as the colour pattern produced must be fixed beforehand.
\thref{lem:absorb} guarantees the existence of such an absorbing structure within our graph collection.

\begin{restatable}[Absorbing $k$-path]{lem}{absorb}\thlabel{lem:absorb}
    Let $k\ge 2$ and $0 < \frac{1}{n} \ll \gamma \ll \beta \ll \alpha \leq 1$. 
    There exists $a \leq 500\beta k n$ such that the following holds.  
    Let $\bG = \{G_1, \ldots, G_m\}$ be a collection of graphs on a common vertex set $V$ with $|V| = n$ where $\delta(\bG) \geq (1-\frac{1}{2k} + \alpha)n$. Let $\chi$ be an edge-colouring of $P_{a+\beta n +2}^k$ with colours from $[m]$.
    Let $Z \subseteq V$ have size $(\beta + \gamma)n + 2$ and $z_1,z_2\in Z$ be distinct vertices. 
    Then there exists an `absorbing' set $A \subseteq V \setminus Z$ of size $a$ such that for every $Z' \subseteq Z \setminus \{z_1,z_2\}$ of size $\beta n$,
    there is a $\chi$-coloured $k$-path from $z_1,u_2,\ldots,u_k$ to $v_1,\ldots,v_{k-1},z_2$ in $\bG$ that covers $A \cup Z'$ and has $u_2,\ldots,u_k,v_1,\ldots,v_{k-1}\in A$.
\end{restatable}

The proof of \thref{lem:absorb} can be found in Section~\ref{Sec:Absorb}. 
Roughly speaking, our absorbing structure will be constructed as a union of absorbing gadgets, 
each of which is a constant size and has low degeneracy so it can be found in our graph collection in a robust way. 
Note that each absorbing gadget will be built on a vertex set $L\cup R$ with two special vertices $x,y\in R$ so that,
for any vertex $v\in L$, there exists a $k$-path from $x$ to $y$ with the same fixed colouring that uses exactly the vertices in $R\cup \{v\}$.
See Figure~\ref{fig:absorber} in Section~\ref{Sec:Absorb} for an illustration of such an absorbing gadget. 
Following a technique first applied by Montgomery~\cite{M}, 
we will build our absorbing structure from these absorbing gadgets according to an auxiliary template given by a robustly matchable bipartite graph. 
Finally, it is worth noting that although the proof of \thref{lem:absorb} follows the same strategy as the proof of \cite[Lemma~3.1]{BMPS}, 
our absorbing gadgets themselves differ from those used by Bowtell, Morris, Pehova, and Staden~\cite{BMPS}, 
in that our gadgets do not reduce to theirs in the case where $k=1$. 

Having set aside this absorbing structure, we will then construct an almost spanning structure composed of many $k$-paths of constant length. 
The existence of this almost spanning collection of $k$-paths is guaranteed by \thref{lem:kpathcollection}. 

\begin{restatable}[$k$-path collection]{lem}{kPathCollection}
    \thlabel{lem:kpathcollection}
        Let $k\in \mathbb{N}$ and $0 < \frac{1}{n} \ll \varepsilon, \frac{1}{r} \ll \alpha \leq 1$. Let $\bG = \{G_1, \ldots, G_{kn}\}$ be a collection of $r$-partite graphs on a common vertex partition
        $V_1, \ldots, V_r$ with $|V_i| = \lfloor\frac{n}{r}\rfloor$ for each $i$, such that $\delta(\bG[V_i, V_j]) \geq (1-\frac{1}{2k} + \alpha)\frac{n}{r}$ for all $i\in [r-1]$ and $i+1\le j\le \min\{i+k,r\}$.
        Let $s \leq (1 -\varepsilon)\frac{n}{r}$. For each $\ell \in [s]$, let $\chi_\ell$ be a colour pattern for $P_r^k$ using colours from $[kn]$.
        Then there exist vertex-disjoint $k$-paths $P_1, \ldots, P_s$ such that $P_i$ is $\chi_i$-coloured for every $i \in [s]$. 
\end{restatable}

The proof of \thref{lem:kpathcollection} is given in Section~\ref{Sec:Cover},
and uses a simple random process to produce the required $k$-paths, one at a time. 
Our strategy is similar to that of Bowtell, Morris, Pehova, and Staden in \cite[Lemma~3.2]{BMPS}.
There, at each step in their algorithm, the authors recursively constructed perfect matchings between consecutive sets in their partition via a simple consequence of Hall's matching theorem, 
before selecting one of the resulting paths uniformly at random.
In that case, any choice of matchings would suffice, as they would necessarily give a path factor in $\bG$ with the required colour pattern. 
However, the situation becomes more complicated when dealing with $k$-paths,
as it is not the case that choosing any matchings between pairs of sets $V_i$ and $V_j$ where $i-j\in[k]$ will necessarily produce a $k$-path factor. 
As a result, we will instead utilize an auxiliary bipartite graph to simultaneously build matchings between each set $V_i$ and all sets $V_j$ with $i-j\in[k]$ in a way that ensures the resulting collection of matchings is a $k$-path factor.
Note that this process will produce the required $k$-path collection provided that, with high probability,
the minimum degree condition between each set in the partition is maintained after each step. 
We will prove that this property holds by analysing our algorithm using a nibble approach pioneered by R\H{o}dl~\cite{R},
which will allow us to deal with several iterations of the algorithm at once. 

Now, having constructed this almost spanning $k$-path collection, we then connect $z_2$ to the start of the first $k$-path,
and each $k$-path to the next, through the reservoir set $Z$. We then extend the resulting $k$-path to $z_1$, ensuring we use up the remaining vertices in $V\setminus A$ (those not covered by the almost spanning $k$-path collection) and all but a small portion of the vertices in $Z$. 
This is made possible by repeated application of \thref{prop:kconnector}, as there are relatively few connections to be made. 
Finally, we complete the required Hamilton $k$-cycle via the appropriately coloured $k$-path from $z_1$ to $z_2$ whose existence is guaranteed by the absorbing structure we set aside at the beginning.

\section{Preliminaries}
\label{Sec:Prelim}

In this section we introduce notation that will be used throughout the paper, along with some preliminary results. 
Given a graph $G$, a vertex $v\in V(G)$, and a set $U\subseteq V(G)$, let $N_G(v,U):=N_G(v)\cap U$ and $d_G(v,U):=|N_G(v,U)|$. 
Let $\bG=\{G_1,\ldots,G_m\}$ be a collection of not necessarily distinct graphs on a common vertex set $V$. Recall that the minimum degree of a graph collection is $\delta(\bG):=\min\{\delta(G_i):i\in [m]\}$.
For $A\subseteq V$, let $\bG[A]:=\{G_1[A],\ldots,G_m[A]\}$. Similarly, for $A,B\subseteq V$, 
let $\bG[A,B]:=\{G_1[A,B],\ldots,G_m[A,B]\}$, where $G_i[A,B]$ is the induced bipartite subgraph of $G_i$ with vertex classes $A$ and $B$ and $E(G_i[A,B]) = \{ab \in E(G_i): a \in A, b \in B\}$. 

We use standard notation for the hierarchy of constants. 
Hence $0 < \alpha \ll \beta \ll \gamma < 1$ means we can choose constants $\alpha, \beta, \gamma$ from right to left while satisfying the desired hierarchy.
More precisely, this means that there exist non-decreasing functions $f$ and $g$ such that any statements we make about these constants hold for all $\alpha \leq f(\beta)$ and $\beta \leq g(\gamma)$. 
Larger hierarchies are defined similarly.

We now present a few basic results that will be useful in future sections.
We begin with the following well-known concentration result for binomial and hypergeometric random variables.

\begin{lem}[Chernoff bounds, {\cite[Corollary 21.7]{FK}}]\thlabel{lem:chernoff}
    Let $\varepsilon > 0$, and let $X$ be a binomially or hypergeometrically distributed random variable with mean $\mu$. 
    Then \[\mathbb{P}(X\le  (1-\varepsilon)\mu)\leq e^{-\varepsilon^2\mu/3}.\]
\end{lem}

The following proposition will allow us to partition the vertex set of our graph collection while maintaining degree conditions into each part.
The proof is a straightforward application of \thref{lem:chernoff}, but we include it for completeness.  

\begin{prop}\thlabel{prop:zexistence}
    Let $k\in \mathbb{N}$, and $\alpha, \beta > 0$. 
    There exists $n_0=n_0(\alpha,\beta,k)$ such that for every $n \geq n_0$ the following holds. 
    Let $\bG = \{G_1, \ldots, G_{kn}\}$ be a collection graphs on common vertex set $V$ with $|V| = n$ and $\delta(\bG) \geq (1-\frac{1}{2k} + \alpha)n$. 
    Let $Z$ be a subset of $V$ of size $\beta n$ chosen uniformly at random. Then for every $i \in [kn]$ and $v \in V$, \[d_{G_i}(v, Z)\geq \left(1-\frac{1}{2k} + \frac{\alpha}{2}\right)|Z| \ \mbox{\ \ and\ \ } \ d_{G_i}(v, V\setminus Z)\geq \left(1-\frac{1}{2k} + \frac{\alpha}{2}\right)|V\setminus Z|.\]
\end{prop}

\begin{proof}
    Fix $i\in [kn]$ and $v\in V$, and let $X:=d_{G_i}(v,Z)$. Note that $X$ is a hypergeometric random variable,
    which contains $\beta n$ success states if $v\notin Z$ and $\beta n-1$ success states if $v\in Z$. Hence, 
    $\Ex[X]\ge d_{G_i}(v)\left(\frac{\beta n-1}{n-1}\right)\ge \left(1-\frac{1}{2k} + \alpha-\frac{1}{\beta n}\right)\beta n$.
    Therefore, letting $\mu:=\Ex[X]$ and $\varepsilon := \frac{\alpha}{4}$, we have \[
    \mathbb{P}\left(X \leq \left(1 - \frac{1}{2k} + \frac{\alpha}{2}\right)|Z|\right) \leq \mathbb{P}\left(X\le (1-\varepsilon)\left(1-\frac{1}{2k}+\alpha-\frac{1}{\beta n}\right)\beta n\right)\le \mathbb{P}\left(X\le (1-\varepsilon)\mu\right),\] where the when first inequality holds when $n$ is sufficiently large. 
    Hence, for sufficiently large $n$, Lemma~\ref{lem:chernoff} implies 
    $\mathbb{P}\left(X \leq \left(1 - \frac{1}{2k} + \frac{\alpha}{2}\right)|Z|\right) \le e^{-\varepsilon^2 \mu/3}\le e^{-\sqrt{n}}$. An analogous argument shows that $\mathbb{P}(d_{G_i}(v, V\setminus Z))\le  e^{-\sqrt{n}}$. 
    Therefore, by taking a union bound, the probability $Z$ has the desired properties is at least $1 - 2ne^{-\sqrt{n}}$, and thus such a set $Z$ exists for sufficiently large $n_0$. 
    \end{proof}

Throughout the proof of \thref{thm:main Gen}, we will frequently need to join a $\chi_1$-coloured $k$-path $P_a^k$ to a $\chi_2$-coloured $k$-path $P_b^k$ so that the $k$-path $P_{a+k+b}^k$ we obtain is appropriately coloured. 
We now show that this can be done via our reservoir whilst avoiding some `small' subset. 

Let $S_1$ and $S_2$ be disjoint $k$-paths with $|S_1|, |S_2| \le k$. 
The $k$\emph{-connector} from $S_1$ to $S_2$ is the graph obtained from the $k$-path on $k + |S_1| + |S_2|$ vertices from $S_1$ to $S_2$ by deleting all edges within $S_1$ and all edges within $S_2$. 
In what follows, we will only need $|S_1|,|S_2| \in \{1,k\}$. For ease of notation, we will use sometimes use $J^k_{a,b}$ to denote a $k$-connector where $|S_1|=a$ and $|S_2|=b$. See Figure \ref{fig:extend+connect} for an illustration.

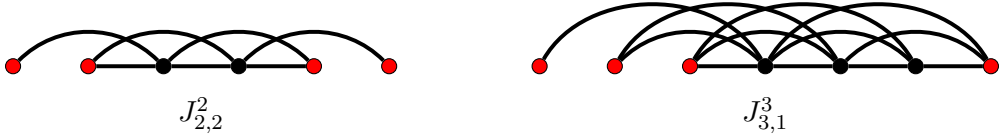
\begin{figure}[h]
\begin{center}
\begin{tikzpicture}

\begin{scope}[shift={(-4,0)}]

\node [] (J) at (-0.5,-0.65) [label= above: {}]{$J^2_{2,2}$};

\node [stdr] (1) at (-3,0) [label= above: {}]{};
\node [stdr] (2) at (-2,0) [label= above: {}]{};
\node [std] (3) at (-1,0) [label= above: {}]{};
\node [std] (4) at (0,0) [label= above: {}]{};
\node [stdr] (5) at (1,0) [label= above: {}]{};
\node [stdr] (6) at (2,0) [label= above: {}]{};

\draw [line width=1.5pt, color=black] (2)--(3);
\draw [line width=1.5pt, color=black] (3)--(4);
\draw [line width=1.5pt, color=black] (4)--(5);

\draw [line width=1.5pt, color=black] (1) to [out=45,in=135] (3);
\draw [line width=1.5pt, color=black] (2) to [out=45,in=135] (4);
\draw [line width=1.5pt, color=black] (3) to [out=45,in=135] (5);
\draw [line width=1.5pt, color=black] (4) to [out=45,in=135] (6);

\end{scope}

\begin{scope}[shift={(4,0)}]

\node [] (J) at (-1,-0.65) [label= above: {}]{$J^3_{3,1}$};

\node [stdr] (0) at (-4,0) [label= above: {}]{};
\node [stdr] (1) at (-3,0) [label= above: {}]{};
\node [stdr] (2) at (-2,0) [label= above: {}]{};
\node [std] (3) at (-1,0) [label= above: {}]{};
\node [std] (4) at (0,0) [label= above: {}]{};
\node [std] (5) at (1,0) [label= above: {}]{};
\node [stdr] (6) at (2,0) [label= above: {}]{};

\draw [line width=1.5pt, color=black] (2)--(3);
\draw [line width=1.5pt, color=black] (3)--(4);
\draw [line width=1.5pt, color=black] (4)--(5);
\draw [line width=1.5pt, color=black] (5)--(6);

\draw [line width=1.5pt, color=black] (1) to [out=45,in=135] (3);
\draw [line width=1.5pt, color=black] (2) to [out=45,in=135] (4);
\draw [line width=1.5pt, color=black] (3) to [out=45,in=135] (5);
\draw [line width=1.5pt, color=black] (4) to [out=45,in=135] (6);

\draw [line width=1.5pt, color=black] (0) to [out=60,in=120] (3);
\draw [line width=1.5pt, color=black] (1) to [out=60,in=120] (4);
\draw [line width=1.5pt, color=black] (2) to [out=60,in=120] (5);
\draw [line width=1.5pt, color=black] (3) to [out=60,in=120] (6);

\end{scope}

\end{tikzpicture}
\caption{A $2$-connector between two pairs of vertices (left) and a $3$-connector from three vertices to a single vertex (right).}
\label{fig:extend+connect}
\end{center}
\end{figure}

The following result provides us with the mechanics to join coloured $k$-paths by appropriately coloured $k$-connectors.

\begin{prop}\thlabel{prop:kconnector}
    Let $k\in \mathbb{N}$ and $\alpha > 0$. There exists $n_0:=n_0(\alpha,k)$ such that for all $n \geq n_0$ the following holds.
    Let $\bG=\{G_1, \dots, G_m\}$ be a collection of graphs on vertex set $V$ with $|V| = n$.
    Let $Z \subseteq V$ with $|Z|\ge \frac{1}{\alpha}$ such that $d_{G_i}(v, Z) \geq (1-\frac{1}{2k} +\frac{\alpha}{2})|Z|$ for all $v \in V$ and $i \in [m]$.
    Let $U \subseteq V$ satisfy $|U \cap Z| < \alpha |Z|$.
    Let $\mathbf{w}$ and $\mathbf{y}$ be $k$-paths in $\bG[V \setminus Z]$ each on at most $k$ vertices.
    Then for every colour pattern $\chi$, and every choice of there exists a $\chi$-coloured $k$-connector from $\mathbf{w}$ to $\mathbf{y}$ whose internal vertices are in $Z\setminus U$.

\end{prop}

\begin{proof} We may assume $\mathbf{w}$ and $\mathbf{y}$ each have $k$-vertices, as this will imply the result for any shorter paths. 
Let $\mathbf{w}=x_1,...,x_k$ and $\mathbf{y}=x_{2k+1},...,x_{3k}$. 
We will greedily add vertices $x_{k+1},...,x_{2k}$ from $Z \setminus U$, one at a time, to build the required $k$-connector. 
Note that, for any distinct vertices $v_1,\ldots,v_{2k}\in V$ and multiset of colours $\{i_1,\ldots, i_{2k}\} \subseteq [m]$, the degree conditions on $Z$ give that 
\[\left|\bigcap_{j=1}^{2k} N_{G_{i_j}}(v_j,Z)\right|\geq \left(1-\frac{\ell}{2k}+\frac{\ell\alpha}{2}\right)|Z|\geq k\alpha|Z|,\] where the final inequality follows by observing that the minimum degree condition implies $\alpha \leq \frac{1}{k}$.
Now, at each step, the new vertex $x$ we wish to add will be a neighbour of at most $2k$ already present vertices.
Therefore, since $|U\cap Z|<\alpha|Z|$ and $\alpha|Z|\geq 1$, and since at most $k-1$ vertices of $X$ have already been chosen,
it is always possible to choose a vertex to play the role of $x$. 
\end{proof}

\section{Main result}
\label{Sec:Main}

The goal of this section is to use \thref{lem:absorb,lem:kpathcollection} to prove \thref{thm:main Gen},
which we restate below for ease of reference.

\MainTheoremGen*

\begin{proof}[Proof of \thref{thm:main Gen}:]
    Set $0 < 
    \frac{1}{n_0} \ll \varepsilon, \frac{1}{r} \ll \gamma \ll \beta \ll \alpha$. 
   It suffices to consider the case where $\chi$ is a bijection (and so $m=kn$). 
   We will construct a $\chi$-coloured Hamilton $k$-cycle $C=x_1,\ldots,x_n$.

    Let $Z$ be a subset of $V$ of size $(\beta + \gamma) n + 2$ chosen uniformly at random. By \thref{prop:zexistence},
    for all $i\in [kn]$ and $v\in V$, we have $d_{G_{i}}(v, Z) \geq \left(1-\frac{1}{2k} + \frac{\alpha}{2}\right)|Z|$ and $d_{G_{i}}(v, V\setminus Z) \geq \left(1-\frac{1}{2k} + \frac{\alpha}{2}\right)|V\setminus Z|$. 
    Let $z_1, z_2 \subseteq Z$ be distinct vertices chosen arbitrarily.
    Let $a\leq 500\beta kn$ be the constant given by \thref{lem:absorb} with constants $0<\frac{1}{n}\ll \gamma\ll \beta \ll \alpha \le 1$, and define $m:=a+\beta n+2$. 
    Let $x_1:=z_1$ and $x_{m}:=z_2$. Let $\chi_0$ be the restriction of $\chi$ to the copy of $P_{m}^k$ induced on vertex set $[m]$ of $V(C^{k}_{n})$.    
    By \thref{lem:absorb}, we can find $A \subseteq V \setminus Z$ of size $a$ such that, for every subset $Z' \subseteq Z\setminus \{z_1, z_2\}$ of size $\beta n$,
    there is a $\chi_0$-coloured $k$-path on $m$ vertices from $z_1$ to $z_2$ using exactly the vertices in $A \cup Z'\cup\{z_1,z_2\}$. 
    Moreover, by the proof of \thref{lem:absorb}, regardless of what subset $Z'\subseteq Z\setminus \{z_1, z_2\}$ is chosen, 
    the resulting $k$-path will always have the same initial $k$ vertices, say $z_1,u_2,\ldots,u_k$, as well as the same final $k$ vertices, say $y_1,\ldots,y_{k-1},z_2$, all of which are in $A\cup\{z_1,z_2\}$.

   We now wish to find an appropriately coloured $k$-path from $z_2$ to $z_1$ in $V\setminus A$ that covers $V \setminus (A \cup Z)$ and
    uses up all but $\beta n$ vertices of $Z$.
     Define $V' := V\setminus (A \cup Z)$ and $n' := |V'|$, noting that $n'=(1-\beta-\gamma)n-a-2k$. Let $\bG' := \{G_1', \ldots, G_{kn}'\}$, where $G_i' := G_i[V']$ for every $i \in [kn]$. 
    Note that $\delta(\bG') \geq (1-\frac{1}{2k} + \frac{\alpha}{2})n$ since $\beta, \gamma \ll \alpha$. 
    
    We will first cover almost all vertices of $V'$ with appropriately coloured $k$-paths of length $r$. 
    We then connect these $k$-paths together through $Z$ in a way that covers the remaining vertices of $V'$ whilst leaving precisely $\beta n$ vertices of $Z$ unused. Let $(V_0, \ldots, V_r)$ be a  partition of $V'$ into $r$ sets $V_1,\ldots,V_r$ of size $\lfloor \frac{n'}{r} \rfloor$ and one set $V_0$ of size $n' - r\lfloor \frac{n'}{r} \rfloor \leq r$ chosen uniformly at random.
    By \thref{prop:zexistence}, $\delta(\bG'[V_i, V_{j}]) \geq (1-\frac{1}{2k} + \frac{\alpha}{4})\lfloor\frac{n'}{r}\rfloor$ for all $i\not= j\in [r]$.

    Let $s := \lfloor(1-\varepsilon)\lfloor \frac{n'}{r} \rfloor\rfloor$. For each $i\in[s]$, let $\chi_i$ be the restriction of $\chi$ to the copy of $P_r^k$ induced on vertex set $\{m+i(k+r)-(r-1),\ldots,m+i(k+r)\}$ of $V(C_n^k)$.
    Applying \thref{lem:kpathcollection} with constants $0<\frac{1}{n}\ll \varepsilon,\frac{1}{r}\ll \frac{\alpha}{4}\le 1$ to the partition $V_1,\ldots,V_r$ 
    gives us vertex-disjoint $k$-paths $P_1, \ldots, P_s$ such that $P_i$ is $\chi_i$-coloured for all $i\in [s]$.

    Now, let $\chi'$ be the restriction of $\chi$ to the copy of $P_{n-m}^k$ that $C_n^k$ induces from $m-k+1$ to $n-k$.
    For each $i\in[s]$, define $\chi'_i$ to be the restriction of $\chi$ to the $k$-connector that $C_n^k$ induces from $m+(i-1)r+(i-2)k+1,...,m+(i-1)r+(i-1)k$ to $m+(i-1)r+ik+1,...,m+(i-1)r+(i+1)k$. 
    Finally, let $\chi'_{s+1}$ be the restriction of $\chi$ to the $k$-connector that $C_n^k$ induces from $n-2k+1,...,n-k$ to $1,...,k$. 
    It remains to make the following connections 
 using only vertices in $(V' \cup Z)\setminus\{z_1, z_2\}$ such that each $k$-path and $k$-connector is internally disjoint from any previous ones:
    \begin{itemize}
        \item[(i)] Connect $y_1,\ldots,y_{k-1},z_2$ to the start of $P_1$ via a $\chi'_1$-coloured $k$-connector.
        \item[(ii)] For each $i\in[s-1]$, connect the end of $P_i$ start of $P_{i+1}$ via a $\chi'_{i+1}$-coloured $k$-connector.
        \item[(iii)] Connect the end of $P_s$ to $(z_1,u_2,\ldots,u_k)$ so that the resulting $k$-path is $\chi'$-coloured and covers any remaining vertices in $V'$.  
    \end{itemize}

    First, note that the total number of connections in (i) and (ii) is at most $\frac{n'}{r} \ll \gamma n$.
Therefore, by repeatedly applying \thref{prop:kconnector}, we can make these connections greedily through $Z$.
It remains to make the connection in (iii). The number of vertices that still need to be covered in $V'$ is $c := n' - sr \ll \gamma n$. 
Let $v_1,\ldots,v_c$ be an arbitrary ordering on the remaining vertices in $V'$.
By repeatedly applying \thref{prop:kconnector}, we can greedily join the end of $P_s$ to $v_1$ through $Z$ via a $k$-connector in the required colours,
then join the end of the resulting $k$-path to $v_2$ via another $k$-connector in the required colours, and so on. 
Note that this process terminates with a $k$-path from $z_2$ to $v_c$ with the required colour pattern, which covers all of $V'$, 
has internal vertices in $V'\cup Z$, and contains $kc\ll\gamma n$ vertices from $Z\setminus\{z_1,z_2\}$. 
Let the last $k$ vertices of this path be $x_1,\ldots,x_{k-1},v_c$ where $x_i\in Z\setminus\{z_1,z_2\}$ for all $i\in[k-1]$.
    
To complete step (iii), 
we must join $x_1,\ldots,x_{k-1},v_c$ to $z_1,u_2,\ldots,u_k$ using only vertices that remain in $Z\setminus\{z_1,z_2\}$.
To do this we will greedily and recursively extend our $k$-path, one vertex at a time, 
until exactly $\beta n+k$ vertices of $Z\setminus\{z_1,z_2\}$ remain. 
Since the number of vertices in $Z\setminus\{z_1,z_2\}$ that have already been used is $ck\ll \gamma n$, and since $\gamma\ll\beta\ll\alpha$,
the minimum degree into $Z$ remains sufficient throughout this greedy process.
We then apply \thref{prop:kconnector} one last time to connect the end of the resulting $k$-path to $z_1,u_2,\ldots,u_k$ via a $\chi'_{s+1}$-coloured $k$-connector,
leaving a set $Z' \subseteq Z \setminus\{z_1, z_2\}$ of exactly $\beta n$ vertices uncovered.

Finally, to complete the desired Hamilton $k$-cycle, we take the  $\chi_0$-coloured $k$-path whose existence is guaranteed by the absorbing property of $A$ given by \thref{lem:absorb}.\end{proof}

\section{Absorption}
\label{Sec:Absorb}

Our goal in this section is to prove \thref{lem:absorb}, which we restate here for convenience.

\absorb*

As outlined in Section~\ref{Sec:Overview}, our strategy will be to construct an absorbing structure that provides a rainbow $k$-path (with a fixed colour pattern) avoiding any small given set of vertices.
In particular, this absorbing structure will always produce the same segment of our Hamilton $k$-cycle, regardless of which vertices are to be avoided. 
The absorbing structure will be composed of relatively few disjoint absorbing gadgets, each of which has a constant size and low degeneracy so it can be found in our graph collection in a robust way. 
We will then utilize a robustly matchable bipartite graph as a template to facilitate the construction of our absorbing structure from these gadgets. 

We begin by defining our absorbing gadgets. Let $k,\ell,m\in \mathbb{N}$ with $\ell\ge 2$. Let $\chi:E(P_{r}^k)\rightarrow[m]$ where $r:=(2k+1)\ell$. 
For a given set $L$ of $\ell$ vertices, an edge-coloured graph $F$ with vertex set $R\cup L$ is an $(L,\chi)$\emph{-absorbing gadget} if, for every $v \in L$, there exists a $\chi$-coloured $k$-path on vertex set $R\cup\{v\}$ whose first $k$ and final $k$ vertices lie in $R$.
We begin by providing an explicit construction of such an absorbing gadget. 
See Figure~\ref{fig:absorber} for an illustration when $k=2$ and $\ell=4$.

\vspace{3mm}

\begin{figure}[h]
\begin{center}
\begin{tikzpicture}

\begin{scope}
\definecolor{brinkpink}{rgb}{0.98, 0.38, 0.5}
\definecolor{pastelpink}{rgb}{1.0, 0.82, 0.86}
\definecolor{pansypurple}{rgb}{0.47, 0.09, 0.29}

\draw [line width=2pt, color=black] (1,1)-- (2,1);
\draw [line width=2pt, color=black] (2,1)-- (3,1);
\draw [line width=2pt, color=black] (3,1)-- (4,1);
\draw [line width=2pt, color=black] (4,1)-- (5,1);
\draw [line width=2pt, color=black] (5,1)-- (6,1);
\draw [line width=2pt, color=black] (6,1)-- (7,1);
\draw [line width=2pt, color=black] (7,1)-- (8,1);
\draw [line width=2pt, color=black] (8,1)-- (9,1);
\draw [line width=2pt, color=black] (9,1)-- (10,1);
\draw [line width=2pt, color=black] (10,1)-- (11,1);
\draw [line width=2pt, color=black] (11,1)-- (12,1);
\draw [line width=2pt, color=black] (12,1)-- (13,1);
\draw [line width=2pt, color=black] (13,1)-- (14,1);
\draw [line width=2pt, color=black] (14,1)-- (15,1);
\draw [line width=2pt, color=black] (15,1)-- (16,1);

\draw [line width=2pt, color=black] (3,1) arc [start angle=180, end angle=0, x radius=1, y radius=0.4];
\draw [line width=2pt, color=black] (4,1) arc [start angle=180, end angle=0, x radius=1, y radius=0.4];
\draw [line width=2pt, color=black] (7,1) arc [start angle=180, end angle=0, x radius=1, y radius=0.4];
\draw [line width=2pt, color=black] (8,1) arc [start angle=180, end angle=0, x radius=1, y radius=0.4];
\draw [line width=2pt, color=black] (11,1) arc [start angle=180, end angle=0, x radius=1, y radius=0.4];
\draw [line width=2pt, color=black] (12,1) arc [start angle=180, end angle=0, x radius=1, y radius=0.4];

\draw [line width=2pt, color=red] (1,1)-- (2.5,-1);
\draw [line width=2pt, color=Orange] (2,1)-- (2.5,-1);
\draw [line width=2pt, color=Dandelion] (3,1)-- (2.5,-1);
\draw [line width=2pt, color=Yellow] (4,1)-- (2.5,-1);

\draw [line width=2pt, color=YellowGreen] (5,1)-- (6.5,-1);
\draw [line width=2pt, color=OliveGreen] (6,1)-- (6.5,-1);
\draw [line width=2pt, color=PineGreen] (7,1)-- (6.5,-1);
\draw [line width=2pt, color=BlueGreen] (8,1)-- (6.5,-1);

\draw [line width=2pt, color=ProcessBlue] (9,1)-- (10.5,-1);
\draw [line width=2pt, color=NavyBlue] (10,1)-- (10.5,-1);
\draw [line width=2pt, color=MidnightBlue] (11,1)-- (10.5,-1);
\draw [line width=2pt, color=CadetBlue] (12,1)-- (10.5,-1);

\draw [line width=2pt, color=Orchid] (13,1)-- (14.5,-1);
\draw [line width=2pt, color=Thistle] (14,1)-- (14.5,-1);
\draw [line width=2pt, color=pastelpink] (15,1)-- (14.5,-1);
\draw [line width=2pt, color=brinkpink] (16,1)-- (14.5,-1);

\draw [line width=2pt, color=red] (1,1)-- (4.5,3);
\draw [line width=2pt, color=Orange] (2,1)-- (4.5,3);
\draw [line width=2pt, color=Dandelion] (3,1)-- (4.5,3);
\draw [line width=2pt, color=Yellow] (4,1)-- (4.5,3);
\draw [line width=2pt, color=YellowGreen] (5,1)-- (4.5,3);
\draw [line width=2pt, color=OliveGreen] (6,1)-- (4.5,3);
\draw [line width=2pt, color=PineGreen] (7,1)-- (4.5,3);
\draw [line width=2pt, color=BlueGreen] (8,1)-- (4.5,3);

\draw [line width=2pt, color=YellowGreen] (5,1)-- (8.5,3);
\draw [line width=2pt, color=OliveGreen] (6,1)-- (8.5,3);
\draw [line width=2pt, color=PineGreen] (7,1)-- (8.5,3);
\draw [line width=2pt, color=BlueGreen] (8,1)-- (8.5,3);
\draw [line width=2pt, color=ProcessBlue] (9,1) -- (8.5,3);
\draw [line width=2pt, color=NavyBlue] (10,1)-- (8.5,3);
\draw [line width=2pt, color=MidnightBlue] (11,1)-- (8.5,3);
\draw [line width=2pt, color=CadetBlue] (12,1)-- (8.5,3);

\draw [line width=2pt, color=ProcessBlue] (9,1) -- (12.5,3);
\draw [line width=2pt, color=NavyBlue] (10,1)-- (12.5,3);
\draw [line width=2pt, color=MidnightBlue] (11,1)-- (12.5,3);
\draw [line width=2pt, color=CadetBlue] (12,1)-- (12.5,3);
\draw [line width=2pt, color=Orchid] (13,1)-- (12.5,3);
\draw [line width=2pt, color=Thistle] (14,1)-- (12.5,3);
\draw [line width=2pt, color=pastelpink] (15,1)-- (12.5,3);
\draw [line width=2pt, color=brinkpink] (16,1)-- (12.5,3);

\begin{scriptsize}
\node [std] (b1) at (1,1)[label=below left: {$b_1$}]{};
\node [std] (b2) at (2,1)[label=below left: {$b_2$}]{};
\node [std] (b3) at (3,1)[label=below right: {$b_3$}]{};
\node [std] (b4) at (4,1)[label=below right: {$b_4$}]{};
\node [std] (b5) at (5,1)[label=below: {$b_5$}]{};
\node [std] (b6) at (6,1)[label=below left: {$b_6$}]{};
\node [std] (b7) at (7,1)[label=below right: {$b_7$}]{};
\node [std] (b8) at (8,1)[label=below right: {$b_8$}]{};
\node [std] (b9) at (9,1)[label=below: {$b_9$}]{};
\node [std] (b10) at (10,1)[label=below left: {$b_{10}$}]{};
\node [std] (b11) at (11,1)[label=below right: {$b_{11}$}]{};
\node [std] (b12) at (12,1)[label=below right: {$b_{12}$}]{};
\node [std] (b13) at (13,1)[label=below: {$b_{13}$}]{};
\node [std] (b14) at (14,1)[label=below left: {$b_{14}$}]{};
\node [std] (b15) at (15,1)[label=below right: {$b_{15}$}]{};
\node [std] (b16) at (16,1)[label=below right: {$b_{16}$}]{};

\node [std] (a1) at (2.5,-1)[label=below: {$a_1$}]{};
\node [std] (a2) at (6.5,-1)[label=below: {$a_2$}]{};
\node [std] (a3) at (10.5,-1)[label=below: {$a_3$}]{};
\node [std] (a4) at (14.5,-1)[label=below: {$a_4$}]{};

\node [std] (c1) at (4.5,3)[label=above: {$c_1$}]{};
\node [std] (c2) at (8.5,3)[label=above: {$c_2$}]{};
\node [std] (c3) at (12.5,3) [label=above: {$c_3$}]{};
\end{scriptsize}
\end{scope}

\begin{scope}[shift={(0,-4)},scale=0.5]
\definecolor{brinkpink}{rgb}{0.98, 0.38, 0.5}
\definecolor{pastelpink}{rgb}{1.0, 0.82, 0.86}
\definecolor{pansypurple}{rgb}{0.47, 0.09, 0.29}

\draw [line width=1.5pt, color=black] (1,1)-- (2,1);
\draw [line width=1.5pt, color=black] (2,1)-- (3,1);
\draw [line width=1.5pt, color=black] (3,1)-- (4,1);
\draw [line width=1.5pt, color=black] (4,1)-- (5,1);
\draw [line width=1.5pt, color=black] (5,1)-- (6,1);
\draw [line width=1.5pt, color=black] (6,1)-- (7,1);
\draw [line width=1.5pt, color=black] (7,1)-- (8,1);
\draw [line width=1.5pt, color=black] (8,1)-- (9,1);
\draw [line width=1.5pt, color=black] (9,1)-- (10,1);
\draw [line width=1.5pt, color=black] (10,1)-- (11,1);
\draw [line width=1.5pt, color=black] (11,1)-- (12,1);
\draw [line width=1.5pt, color=black] (12,1)-- (13,1);
\draw [line width=1.5pt, color=black] (13,1)-- (14,1);
\draw [line width=1.5pt, color=black] (14,1)-- (15,1);
\draw [line width=1.5pt, color=black] (15,1)-- (16,1);

\draw [line width=1.5pt, color=black] (3,1) arc [start angle=180, end angle=0, x radius=1, y radius=0.4];
\draw [line width=1.5pt, color=black] (4,1) arc [start angle=180, end angle=0, x radius=1, y radius=0.4];
\draw [line width=1.5pt, color=black] (7,1) arc [start angle=180, end angle=0, x radius=1, y radius=0.4];
\draw [line width=1.5pt, color=black] (8,1) arc [start angle=180, end angle=0, x radius=1, y radius=0.4];
\draw [line width=1.5pt, color=black] (11,1) arc [start angle=180, end angle=0, x radius=1, y radius=0.4];
\draw [line width=1.5pt, color=black] (12,1) arc [start angle=180, end angle=0, x radius=1, y radius=0.4];

\draw [line width=1.5pt, color=red] (1,1)-- (2.5,-1);
\draw [line width=1.5pt, color=Orange] (2,1)-- (2.5,-1);
\draw [line width=1.5pt, color=Dandelion] (3,1)-- (2.5,-1);
\draw [line width=1.5pt, color=Yellow] (4,1)-- (2.5,-1);

\draw [line width=1.5pt, color=Gray, dashed] (5,1)-- (6.5,-1);
\draw [line width=1.5pt, color=Gray, dashed] (6,1)-- (6.5,-1);
\draw [line width=1.5pt, color=Gray, dashed] (7,1)-- (6.5,-1);
\draw [line width=1.5pt, color=Gray, dashed] (8,1)-- (6.5,-1);

\draw [line width=1.5pt, color=Gray, dashed] (9,1)-- (10.5,-1);
\draw [line width=1.5pt, color=Gray, dashed] (10,1)-- (10.5,-1);
\draw [line width=1.5pt, color=Gray, dashed] (11,1)-- (10.5,-1);
\draw [line width=1.5pt, color=Gray, dashed] (12,1)-- (10.5,-1);

\draw [line width=1.5pt, color=Gray, dashed] (13,1)-- (14.5,-1);
\draw [line width=1.5pt, color=Gray, dashed] (14,1)-- (14.5,-1);
\draw [line width=1.5pt, color=Gray, dashed] (15,1)-- (14.5,-1);
\draw [line width=1.5pt, color=Gray, dashed] (16,1)-- (14.5,-1);

\draw [line width=1.5pt, color=Gray, dashed] (1,1)-- (4.5,3);
\draw [line width=1.5pt, color=Gray, dashed] (2,1)-- (4.5,3);
\draw [line width=1.5pt, color=Gray, dashed] (3,1)-- (4.5,3);
\draw [line width=1.5pt, color=Gray, dashed] (4,1)-- (4.5,3);
\draw [line width=1.5pt, color=YellowGreen] (5,1)-- (4.5,3);
\draw [line width=1.5pt, color=OliveGreen] (6,1)-- (4.5,3);
\draw [line width=1.5pt, color=PineGreen] (7,1)-- (4.5,3);
\draw [line width=1.5pt, color=BlueGreen] (8,1)-- (4.5,3);

\draw [line width=1.5pt, color=Gray, dashed] (5,1)-- (8.5,3);
\draw [line width=1.5pt, color=Gray, dashed] (6,1)-- (8.5,3);
\draw [line width=1.5pt, color=Gray, dashed] (7,1)-- (8.5,3);
\draw [line width=1.5pt, color=Gray, dashed] (8,1)-- (8.5,3);
\draw [line width=1.5pt, color=ProcessBlue] (9,1) -- (8.5,3);
\draw [line width=1.5pt, color=NavyBlue] (10,1)-- (8.5,3);
\draw [line width=1.5pt, color=MidnightBlue] (11,1)-- (8.5,3);
\draw [line width=1.5pt, color=CadetBlue] (12,1)-- (8.5,3);

\draw [line width=1.5pt, color=Gray, dashed] (9,1) -- (12.5,3);
\draw [line width=1.5pt, color=Gray, dashed] (10,1)-- (12.5,3);
\draw [line width=1.5pt, color=Gray, dashed] (11,1)-- (12.5,3);
\draw [line width=1.5pt, color=Gray, dashed] (12,1)-- (12.5,3);
\draw [line width=1.5pt, color=Orchid] (13,1)-- (12.5,3);
\draw [line width=1.5pt, color=Thistle] (14,1)-- (12.5,3);
\draw [line width=1.5pt, color=pastelpink] (15,1)-- (12.5,3);
\draw [line width=1.5pt, color=brinkpink] (16,1)-- (12.5,3);

\begin{scriptsize}
\node [stdsm] (b1) at (1,1){};
\node [stdsm] (b2) at (2,1){};
\node [stdsm] (b3) at (3,1){};
\node [stdsm] (b4) at (4,1){};
\node [stdsm] (b5) at (5,1){};
\node [stdsm] (b6) at (6,1){};
\node [stdsm] (b7) at (7,1){};
\node [stdsm] (b8) at (8,1){};
\node [stdsm] (b9) at (9,1){};
\node [stdsm] (b10) at (10,1){};
\node [stdsm] (b11) at (11,1){};
\node [stdsm] (b12) at (12,1){};
\node [stdsm] (b13) at (13,1){};
\node [stdsm] (b14) at (14,1){};
\node [stdsm] (b15) at (15,1){};
\node [stdsm] (b16) at (16,1){};

\node [stdsm] (a1) at (2.5,-1){};
\node [stdsm] (a2) at (6.5,-1){};
\node [stdsm] (a3) at (10.5,-1){};
\node [stdsm] (a4) at (14.5,-1){};

\node [stdsm] (c1) at (4.5,3){};
\node [stdsm] (c2) at (8.5,3){};
\node [stdsm] (c3) at (12.5,3){};
\end{scriptsize}
\end{scope}

\begin{scope}[shift={(8,-4)},scale=0.5]
\definecolor{brinkpink}{rgb}{0.98, 0.38, 0.5}
\definecolor{pastelpink}{rgb}{1.0, 0.82, 0.86}
\definecolor{pansypurple}{rgb}{0.47, 0.09, 0.29}

\draw [line width=1.5pt, color=black] (1,1)-- (2,1);
\draw [line width=1.5pt, color=black] (2,1)-- (3,1);
\draw [line width=1.5pt, color=black] (3,1)-- (4,1);
\draw [line width=1.5pt, color=black] (4,1)-- (5,1);
\draw [line width=1.5pt, color=black] (5,1)-- (6,1);
\draw [line width=1.5pt, color=black] (6,1)-- (7,1);
\draw [line width=1.5pt, color=black] (7,1)-- (8,1);
\draw [line width=1.5pt, color=black] (8,1)-- (9,1);
\draw [line width=1.5pt, color=black] (9,1)-- (10,1);
\draw [line width=1.5pt, color=black] (10,1)-- (11,1);
\draw [line width=1.5pt, color=black] (11,1)-- (12,1);
\draw [line width=1.5pt, color=black] (12,1)-- (13,1);
\draw [line width=1.5pt, color=black] (13,1)-- (14,1);
\draw [line width=1.5pt, color=black] (14,1)-- (15,1);
\draw [line width=1.5pt, color=black] (15,1)-- (16,1);

\draw [line width=1.5pt, color=black] (3,1) arc [start angle=180, end angle=0, x radius=1, y radius=0.4];
\draw [line width=1.5pt, color=black] (4,1) arc [start angle=180, end angle=0, x radius=1, y radius=0.4];
\draw [line width=1.5pt, color=black] (7,1) arc [start angle=180, end angle=0, x radius=1, y radius=0.4];
\draw [line width=1.5pt, color=black] (8,1) arc [start angle=180, end angle=0, x radius=1, y radius=0.4];
\draw [line width=1.5pt, color=black] (11,1) arc [start angle=180, end angle=0, x radius=1, y radius=0.4];
\draw [line width=1.5pt, color=black] (12,1) arc [start angle=180, end angle=0, x radius=1, y radius=0.4];

\draw [line width=1.5pt, color=Gray, dashed] (1,1)-- (2.5,-1);
\draw [line width=1.5pt, color=Gray, dashed] (2,1)-- (2.5,-1);
\draw [line width=1.5pt, color=Gray, dashed] (3,1)-- (2.5,-1);
\draw [line width=1.5pt, color=Gray, dashed] (4,1)-- (2.5,-1);

\draw [line width=1.5pt, color=YellowGreen] (5,1)-- (6.5,-1);
\draw [line width=1.5pt, color=OliveGreen] (6,1)-- (6.5,-1);
\draw [line width=1.5pt, color=PineGreen] (7,1)-- (6.5,-1);
\draw [line width=1.5pt, color=BlueGreen] (8,1)-- (6.5,-1);

\draw [line width=1.5pt, color=Gray, dashed] (9,1)-- (10.5,-1);
\draw [line width=1.5pt, color=Gray, dashed] (10,1)-- (10.5,-1);
\draw [line width=1.5pt, color=Gray, dashed] (11,1)-- (10.5,-1);
\draw [line width=1.5pt, color=Gray, dashed] (12,1)-- (10.5,-1);

\draw [line width=1.5pt, color=Gray, dashed] (13,1)-- (14.5,-1);
\draw [line width=1.5pt, color=Gray, dashed] (14,1)-- (14.5,-1);
\draw [line width=1.5pt, color=Gray, dashed] (15,1)-- (14.5,-1);
\draw [line width=1.5pt, color=Gray, dashed] (16,1)-- (14.5,-1);

\draw [line width=1.5pt, color=red] (1,1)-- (4.5,3);
\draw [line width=1.5pt, color=Orange] (2,1)-- (4.5,3);
\draw [line width=1.5pt, color=Dandelion] (3,1)-- (4.5,3);
\draw [line width=1.5pt, color=Yellow] (4,1)-- (4.5,3);
\draw [line width=1.5pt, color=Gray, dashed] (5,1)-- (4.5,3);
\draw [line width=1.5pt, color=Gray, dashed] (6,1)-- (4.5,3);
\draw [line width=1.5pt, color=Gray, dashed] (7,1)-- (4.5,3);
\draw [line width=1.5pt, color=Gray, dashed] (8,1)-- (4.5,3);

\draw [line width=1.5pt, color=Gray, dashed] (5,1)-- (8.5,3);
\draw [line width=1.5pt, color=Gray, dashed] (6,1)-- (8.5,3);
\draw [line width=1.5pt, color=Gray, dashed] (7,1)-- (8.5,3);
\draw [line width=1.5pt, color=Gray, dashed] (8,1)-- (8.5,3);
\draw [line width=1.5pt, color=ProcessBlue] (9,1) -- (8.5,3);
\draw [line width=1.5pt, color=NavyBlue] (10,1)-- (8.5,3);
\draw [line width=1.5pt, color=MidnightBlue] (11,1)-- (8.5,3);
\draw [line width=1.5pt, color=CadetBlue] (12,1)-- (8.5,3);

\draw [line width=1.5pt, color=Gray, dashed] (9,1) -- (12.5,3);
\draw [line width=1.5pt, color=Gray, dashed] (10,1)-- (12.5,3);
\draw [line width=1.5pt, color=Gray, dashed] (11,1)-- (12.5,3);
\draw [line width=1.5pt, color=Gray, dashed] (12,1)-- (12.5,3);
\draw [line width=1.5pt, color=Orchid] (13,1)-- (12.5,3);
\draw [line width=1.5pt, color=Thistle] (14,1)-- (12.5,3);
\draw [line width=1.5pt, color=pastelpink] (15,1)-- (12.5,3);
\draw [line width=1.5pt, color=brinkpink] (16,1)-- (12.5,3);

\begin{scriptsize}
\node [stdsm] (b1) at (1,1){};
\node [stdsm] (b2) at (2,1){};
\node [stdsm] (b3) at (3,1){};
\node [stdsm] (b4) at (4,1){};
\node [stdsm] (b5) at (5,1){};
\node [stdsm] (b6) at (6,1){};
\node [stdsm] (b7) at (7,1){};
\node [stdsm] (b8) at (8,1){};
\node [stdsm] (b9) at (9,1){};
\node [stdsm] (b10) at (10,1){};
\node [stdsm] (b11) at (11,1){};
\node [stdsm] (b12) at (12,1){};
\node [stdsm] (b13) at (13,1){};
\node [stdsm] (b14) at (14,1){};
\node [stdsm] (b15) at (15,1){};
\node [stdsm] (b16) at (16,1){};

\node [stdsm] (a1) at (2.5,-1){};
\node [stdsm] (a2) at (6.5,-1){};
\node [stdsm] (a3) at (10.5,-1){};
\node [stdsm] (a4) at (14.5,-1){};

\node [stdsm] (c1) at (4.5,3){};
\node [stdsm] (c2) at (8.5,3){};
\node [stdsm] (c3) at (12.5,3){};

\end{scriptsize}
\end{scope}

\end{tikzpicture}
\caption{The absorbing gadget $F^2_4$ (top), where the black edges all have distinct colours not used elsewhere in the absorber, 
and how it absorbs $a_1$ (bottom left) and $a_2$ (bottom right). The vertices $a_3$ and $a_4$ are absorbed symmetrically.}\label{fig:absorber}
\end{center}
\end{figure}
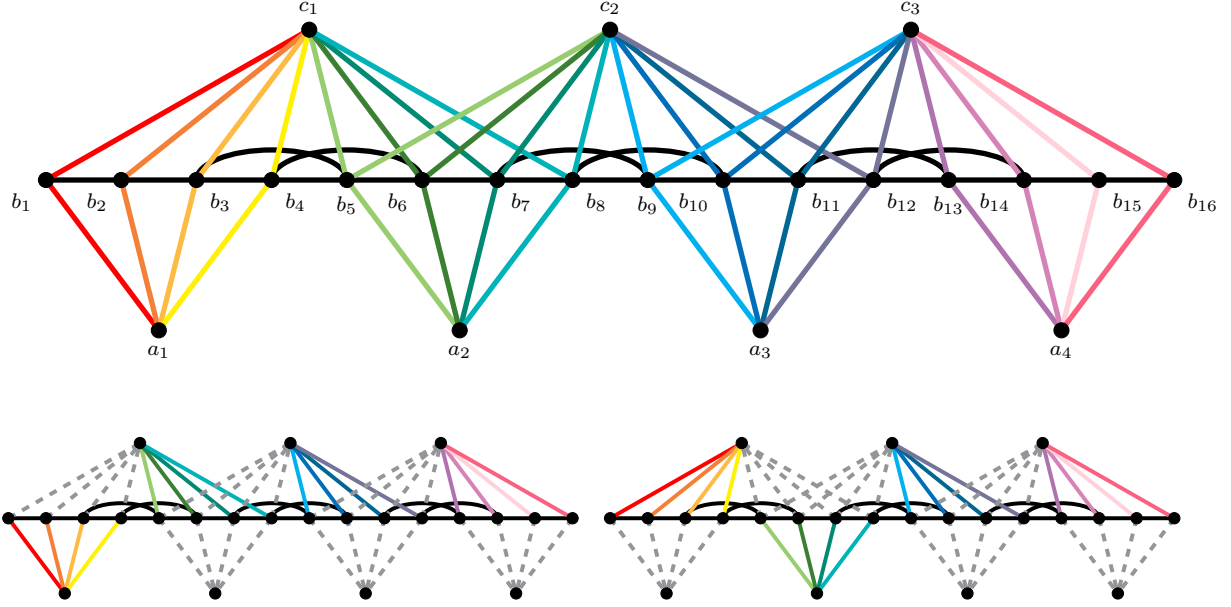

Given $\chi:E(P_{r}^k)\rightarrow[m]$, and disjoint sets of vertices $A = \{a_1, \ldots, a_{\ell}\}$, $B = \{b_1, \ldots, b_{2k\ell}\}$, and
$C = \{c_1, \ldots, c_{\ell-1}\}$, define the edge-coloured graph $F^k_{\ell}$ = $F^k_{\ell}(\chi,A,B,C)$ as follows: 
Let $S$ be the sequence of vertices in $A\cup B$ where the vertices in $B$ appear in order, and for each $i$, the vertex $a_i$ appears directly after $b_{(2i-1)k}$.
Let $F^k_{\ell}[A \cup B]$ be a $\chi$-coloured copy of $P_r^k$ induced on $S$. Now, for each $i\in[\ell]$,
duplicate the vertex $a_i$ twice (call the clones $a'_i$ and $a''_i$), but delete $a'_1$ and $a''_{\ell}$.
Finally, identify the vertices $a''_{i}$ and $a'_{i+1}$ and call this vertex $c_i$ for all $i\in [\ell-1]$. 
Note that $N(c_i) = N(a_{i}) \cup N(a_{i+1})$ and each edge $c_ib$ has the same colour as the unique edge from $\{a_{i},a_{i+1}\}$ to $b$. To see that $F^k_{\ell}$ is an $(A,\chi)$-absorbing gadget with $R = B \cup C$, note that, 
for any $a_i\in A$, the sequence of vertices $S_i$ obtained from $S$ by replacing $a_j$ with $c_j$ for all $1\le j<i$, 
and $a_j$ with $c_{j-1}$ for all $i< j\le \ell$, yields the required $\chi$-coloured $k$-path in $F^k_\ell[B\cup C\cup\{a_i\}]$.

The following lemma will enable us to simultaneously find many disjoint absorbing gadgets in $\bG$,
which will allow us to absorb any subset of our reservoir of a particular size. 

\begin{lem}[Robust existence of absorbing gadgets]\thlabel{lem:robustabsorbers}
    Let $0 < \frac{1}{n} \ll \frac{1}{\ell},\frac{1}{k}, \alpha \leq 1$ where $k,\ell\in \mathbb{N}$ with $\ell\ge 2$. Let $r:=(2k+1)\ell$.
    Let $\bG = \{G_1, \ldots, G_{m}\}$ be a collection of graphs on vertex set $V$ with $|V| = n$ where $\delta(\bG) \geq \left(1-\frac{1}{k+2} + \alpha\right) n$. 
    Let $\chi:E(P_{r}^k)\rightarrow [m]$. 
    Let $U \subseteq V$ where $|U| \leq \frac{\alpha n}{4}$. 
    Then, for every $L \subseteq V$ of size $\ell$, there is an $(L,\chi)$-absorbing gadget in $\bG$ that avoids vertices in $U\setminus L$.
\end{lem}

\thref{lem:robustabsorbers} is a direct consequence of the following result from \cite{BMPS}. 
Say that a graph $G$ has \emph{degeneracy} $p$ if there exists an ordering $v_1,\ldots,v_n$ of the vertices of $G$ such that $d_G(v_i,S_i)\le p$ for all $i\in [n]$, where $S_i=\{v_1,\ldots,v_{i-1}\}$.
In other words, $G$ has degeneracy $p$ if there is an ordering of the vertices such that each vertex has at most $p$ neighbours that come before it in the ordering.
Note that, when we wish to specify the first $i$ vertices in this ordering, we will refer to $G$ as having degeneracy $p$ with initial segment $S_{i+1}$.
Roughly speaking, the following result from \cite{BMPS} guarantees the robust existence of any relatively small graph $J$ with low degeneracy in a graph collection that satisfies the minimum degree requirements. 
As the absorbing gadgets $F_\ell^k$ have sufficiently low degeneracy when $k\ge 2$, we can use this result to prove \thref{lem:robustabsorbers}.

\begin{lem}[{\cite[Lemma~5.4]{BMPS}}]\thlabel{lem:degen}
    Let $0 < \frac{1}{n} \ll \alpha, \frac{1}{p} < 1$ and let $\bG = \{G_1, \ldots, G_m\}$ be a collection of graphs on vertex set $V$ with $|V| = n$. 
    Let $Z \subseteq V(G)$ such that $d_{G_i}(v, Z) \geq \left(\frac{p-1}{p} + \alpha\right)|Z|$ for all $i \in [m]$ and $v \in V$.
    Let $J$ be a graph of order at most $\frac{\alpha}{4}|Z|$, and let $\chi$ be an edge-colouring of $J$ using colours from $[m]$.
    Let $I \subseteq V(J)$ be an independent set in $J$, and let $f: I \to V$ be an injection.
    Suppose $J$ has degeneracy $p$ with initial segment $I$. 
    Let $U \subseteq V\setminus f(I)$ such that $|U \cap Z| \leq \frac{\alpha}{4} |Z|$. 
    Then there is a $\chi$-coloured copy of $J$ avoiding $U$ where $x$ is mapped to $f(x)$ for each $x \in I$, 
    and every vertex in $V(J)\setminus I$ is mapped into $Z$.
\end{lem}

\begin{proof}[Proof of \thref{lem:robustabsorbers}]
     Consider the ordering of $V(F_\ell^k)$ given by $A,s_1,\ldots,s_{\ell-1},b_{2k(\ell - 1)+1},\ldots,b_{2k\ell}$, where $s_i=b_{(2i-2)k+1},\ldots,b_{(2i-1)k},c_i,b_{(2i-1)k+1},\ldots,b_{2ki}$ for all $i\in [\ell-1]$.
     Note that each vertex has at most $k+2$ previous neighbours, and that the vertices in $A$ appear first in the ordering.
     Hence $F_\ell^k$ has degeneracy $k+2$ with initial segment $A$. 
     Therefore, letting $L=A$ and applying \thref{lem:degen} with $Z=V$ and $p=k+2$, 
     it follows that there exists an $(L,\chi)$-absorbing gadget of the form $F_\ell^k$ in $\bG$ avoiding vertices in $U\setminus L$.  
\end{proof}

    Note that \thref{lem:robustabsorbers} does not hold when $k=1$. 
    The reason for this distinction is that $2k<k+2$ when $k=1$, which means our absorbing gadgets do not have the required degeneracy in this case. 
    It is also worth noting that \thref{lem:robustabsorbers} only requires $\delta(\bG) \geq \left(1-\frac{1}{k+2} + \alpha\right) n$.
    However, in order to join these absorbing gadgets together to construct our absorbing structure in the proof of \thref{lem:absorb}, 
    we will still require $\delta(\bG) \geq \left(1-\frac{1}{2k} + \alpha\right) n$.

The last tool we need before proving \thref{lem:absorb} is the following result from \cite{BMPS},
which provides a modified version of the robustly matchable bipartite graph given by Nenadov and Pehova in~\cite{NP}.

\begin{lem}\thlabel{lem:montgomerytemplate}
    \cite[Lemma~5.1]{BMPS} For every $\ep > 0$ and $s \in \mathbb{N}$ sufficiently large, there exists a bipartite graph $B$ with partite sets $U \cup W$ and $X$, where 
    $|U| = 2s$, $|W| = s + \ep s$, $|X| = 3s$, and $2 \leq \delta(B) \leq \Delta(B) \leq 40$, such that for every subset $W' \subseteq W$ with $|W'| = s$, there exists a perfect matching in $B[U \cup W',X]$.
\end{lem} 

We are now ready to prove \thref{lem:absorb}.

\begin{proof}[Proof of \thref{lem:absorb}]
    Let $k\ge 2$ and $0 < \frac{1}{n} \ll \gamma \ll \beta \ll \alpha \leq 1$ such that $s := \beta n$ is an integer.
    Let $B$ be the auxiliary bipartite graph with partite sets $U\cup W$ and $X=\{x_1,...,x_{3s}\}$ that is given by \thref{lem:montgomerytemplate} when $\ep=\frac{\gamma}{\beta}$, and let $b:=|E(B)|$. 
    Define $a := (2k+1)b+(3s+1)k-s$, noting that, since $\Delta(B)\le 40$ and $|X|=3s$, 
    we have $a\le 120(2k+1)s+(3s+1)k-s\le 500 \beta kn$ as required.
    Fix $Z \subseteq V$ of size $(\beta+\gamma)n+2 = s + \varepsilon s + 2$, and let $Y \subseteq V \setminus Z$ have size $2s$. 
    Let $z_1, z_2 \in Z$, and define $\tilde{Z}:=Z\setminus\{z_1,z_2\}$. Identify $Y$ with the set $U$ in the auxiliary bipartite graph $B$, 
    and similarly identify $\tilde{Z}$ with the set $W$ in $B$. 
    We will use $B$ as a template with which to build the absorbing gadgets that will comprise our absorbing structure.

    Partition $\chi$ into disjoint colour patterns $\chi_1,\ldots,\chi_{3s}$ and $\chi'_0,\ldots,\chi'_{3s}$ so that $\chi'_0,\chi_1,\chi'_1,\ldots,\chi_{3s},\chi'_{3s}$ is $\chi$, 
    where $\chi_i$ is a colour pattern for a $k$-path of order $(2k+1)d_B(x_i)+1$ for each $i\in[3s]$, 
    $\chi'_i$ is a colour pattern for $J^k_{k,k}$ for each $i\in[3s-1]$, and $\chi'_0$ and $\chi'_{3s}$ are colour patterns for $J^k_{1,k}$ and $J^k_{k,1}$, respectively. 
    Our goal is to build a $(N_B(x_i),\chi_i)$-absorbing gadget $F_{d_B(x_i)}^k$ for each $i\in[3s]$,
    and then join these gadgets together (and to $z_1$ and $z_2$) via appropriately coloured $k$-connectors to construct our absorbing structure. 
    Note that $d_B(x_i)\ge 2$ for all $i\in[3s]$ since $\delta(B)\ge 2$.

    We construct these absorbing gadgets recursively. 
    Begin by setting $U_0 := Y \cup Z$ and $A_0 := Y$. At step $i\in [3s]$, we apply Lemma~\ref{lem:robustabsorbers} to find a $(N_B(x_i),\chi_i)$-absorbing gadget $F_{d_B(x_i)}^k$ that avoids $U_{i-1}$,
    and then update $U_{i} := U_{i-1} \cup V(F_i)$ and $A_i := (A_{i-1} \cup V(F_i))\setminus Z$.
    Since step $i$ will use exactly $(2k+1)d_B(x_i)-1\le 40(2k+1)$ vertices outside of $Y\cup Z$,
    it follows that 
    \[|U_{i-1}|\le |U_{3s-1}|\le 2s+ 120s(2k+1)=2\beta n+ 120\beta n(2k+1)\le \frac{\alpha n}{4},\]
    and thus we can apply \thref{lem:robustabsorbers} at each step. Note that each absorbing gadget $F_{d_B(x_i)}^k$ has order $(2k+1)d_B(x_i)-1$.  
    For the sake of brevity, for each $i \in [3s]$, let $\mathbf{w}_i$ and $\mathbf{y}_i$ denote the $k$-paths induced by the first $k$ and final $k$ vertices of $F_{\ell_i}^k$, respectively.

Having constructed these absorbing gadgets, it remains to join them together (and to $z_1$ and $z_2$) to form our absorbing structure by iteratively applying \thref{prop:kconnector},
where at each step we use new vertices from $Z$ to create the desired connection. 
The end vertex $z_1$ is joined to $\mathbf{w}_1$ via a $\chi'_0$-coloured $k$-connector $H_0$. 
Similarly, $z_{2}$ is joined to $\mathbf{y}_{3s}$ via a $\chi'_{3s}$-coloured $k$-connector $H_{3s}$. Finally, each internal connection from $\mathbf{y}_i$ to $\mathbf{w}_{i+1}$ is made via a $\chi'_i$-coloured $k$-connector $H_i$.
At each step, $k$ new vertices are used in the connection, so at most $a \le \frac{\alpha}{2}n$ vertices must be avoided at a given time. 
This ensures the remaining vertices in $Z$ maintain the degree conditions required to apply \thref{prop:kconnector}. 
Define $A$ to be the set constructed by adding the vertices of $H_0,\ldots,H_{3s}$ to $A_{3s}$ and note that $|A| = a$.

 It remains to show that $A$ has the required absorbing property. 
 Suppose $Z'$ is a subset of $\tilde{Z}$ of size $\beta n$. 
 Let $M = \{x_1v_1, \ldots, x_{3s}v_{3s}\}$ be the perfect matching in $B$ covering $X\cup Y \cup Z'$ whose existence is guaranteed by \thref{lem:montgomerytemplate}.
 For each $i\in [3s]$, since $F_{d_B(x_i)}^k$ is an $(N_B(x_i),\chi_i)$-absorbing gadget, 
 there is a $\chi_i$-coloured $k$-path $P_i$ from $\mathbf{w}_i$ to $\mathbf{y}_i$ that covers $(V(F_i)\setminus N_B(x_i)) \cup \{v_i\}$.
 Joining these $k$-paths together with the connections constructed above yields a $\chi$-coloured $k$-path $H_0,P_1,H_1,\ldots,P_{3s},H_{3s}$ from $z_1$ to $z_2$ covering $A \cup Z$, as required.\end{proof}

\section{Almost spanning $k$-path collection}
\label{Sec:Cover}

The goal of this section is to prove \thref{lem:kpathcollection}, which we restate here for convenience.

\kPathCollection*

As discussed in Section~\ref{Sec:Overview}, our proof of \thref{lem:kpathcollection} follows a similar strategy to that of Lemma~3.2 in~\cite{BMPS}, 
which uses a random greedy algorithm to generate an almost spanning path collection by recursively constructing perfect matchings via the following simple consequence of Hall's matching theorem.

\begin{lem}\thlabel{lem:matching}
    If $G = (A,B)$ is bipartite with $|A| = |B| = n$ and $\delta(G) \geq \frac{n}{2}$, then $G$ has a perfect matching.
\end{lem}

However, in the case of an almost spanning $k$-path collection, our algorithm will require the following generalization of \thref{lem:matching}.
For ease of notation, given a bipartite graph $G$ with partite sets $A$ and $B$ where $|A|=kn$ and $|B|=n$,
and a perfect $K_k$-tiling $T=\{F_1,...,F_n\}$ of $G[A]$, define $H_{G,T}$ to be the auxiliary bipartite graph with partite sets $X$ and $Y$, where $V(X) = T$ and $V(Y) = B$, 
such that $F_i\in X$ is adjacent to $v\in Y$ if and only if $uv\in E(G)$ for all $u\in V(F_i)$.

\begin{lem}
\thlabel{lem:matchingGen}
Let $k\ge 1$. Let $G$ be a graph of order $(k+1)n$ and $(A,B)$ a partition of $V(G)$ with $|A|=k n$ and $|B|=n$ such that $d_G(v,A) \geq \left(1-\frac{1}{2k}\right)kn$ for all $v\in B$ and $d_G(u,B) \geq \left(1-\frac{1}{2k}\right)n$ for all $u\in A$. 
Then any perfect $K_k$-tiling of $G[A]$ can be extended to a perfect $K_{k+1}$-tiling of $G$.
\end{lem}

\begin{proof}
Let $T=\{F_1,\ldots,F_n\}$ be a perfect $K_k$-tiling of $G[A]$.
Consider the auxiliary bipartite graph $H_{G,T}$ as defined above, noting that $|X|=|Y|=n$.
Now, for any $v\in Y$, since $d_G(v,A) \geq \left(1-\frac{1}{2k}\right)kn$, we have $d_H(v)\ge \left(1-\frac{1}{2k}\right)kn-(k-1)n = \frac{n}{2}$.
In addition, for any $F_i\in T$, since $N_G(u,B) \geq \left(1-\frac{1}{2k}\right)n$ for all $u\in F_i$, we have $d_H(F_i)\ge n-\frac{kn}{2k}= \frac{n}{2}$.
Hence, by \thref{lem:matching}, $H$ has a perfect matching, which, by definition of $E(H)$, 
corresponds to a perfect $K_{k+1}$-tiling of $G$ that extends $T$. 
\end{proof}

We are now ready to prove \thref{lem:kpathcollection}.

\begin{proof}[Proof of \thref{lem:kpathcollection}] 
 We will use a random greedy algorithm (see Algorithm~\ref{A1} on page~\pageref{A1}) to generate the required vertex-disjoint $k$-paths $P_1,\ldots,P_s$, each of order $r$.
 Note that Algorithm~\ref{A1} produces the required collection of $k$-paths provided it does not abort at any step $i\in[s]$. 
 Indeed, at each step, the existence of a perfect matching in the given auxiliary bipartite graph is guaranteed by \thref{lem:matchingGen}, assuming we did not abort.
 Hence it suffices to prove that, with high probability, the algorithm does not abort at any step $i\in[s]$. 
To this end, we will utilize a nibble approach to analyse the algorithm, grouping the steps together into `blocks' of size $\gamma n$, and then proving that each collection behaves well.

\begin{algorithm}[h]
    \caption{$k$-path builder}
        Initialise by fixing  $n_1:=\lfloor n/r\rfloor$ and $V^{1}_j := V_j$ for all $j \in [r]$\;

        \For{$i=1,2,\ldots,s$}{
            Initialize by fixing $T_1:=V_1^i$\;
            \For{$\ell=1,2,\ldots,k-1$}{

                \eIf{$\delta(\bG[V_j^i, V_{\ell+1}^i]) < \frac{(2\ell-1)n_i}{2\ell}$ for some $j\in[\ell]$} 
                    {Abort and return an error\;}
                {Let $G^{\ell+1}$ be the union of $G_{\chi_i(1,\ell+1)}[V_{1}^i, V_{\ell+1}^i],..., G_{\chi_i(\ell,\ell+1)}[V_{\ell}^i, V_{\ell+1}^i]$. 
                Let $\mathcal{M}_{\ell+1}$ be a uniformly random perfect matching in the auxiliary bipartite graph $H_{G^{\ell+1},T_\ell}$. 
                Let $M_{(1,\ell+1)}^i,\ldots,M_{(\ell,\ell+1)}^i$ be the perfect matchings induced by $\mathcal{M}_{\ell+1}$
                in $G_{\chi_i(1,\ell+1)}[V_{1}^i, V_{\ell+1}^i],..., G_{\chi_i(\ell,\ell+1)}[V_{\ell}^i, V_{\ell+1}^i]$, respectively. 
                Let $T_{\ell+1}:=\bigcup_{a=1}^\ell\left(\bigcup_{j=1}^{a} M_{(j,a+1)}^i\right)$\;
                }
            }
            \For{$\ell=k+1,k+2,\ldots,r$}{
                \eIf{$\delta(\bG[V_{\ell-j}^i, V_{\ell}^i]) < \frac{(2k-1)n_i}{2k}$ for some $j\in[k]$}
                    {Abort and return an error\;}
                {Let $G^\ell$ be the union of $G_{\chi_i(\ell-k,\ell)}[V_{\ell-k}^i, V_{\ell}^i],\ldots,G_{\chi_i(\ell-1,\ell)}[V_{\ell-1}^i, V_{\ell}^i]$. Let $\mathcal{M}_{\ell}$ be a uniformly random perfect matching in the auxiliary bipartite graph $H_{G^\ell,T_\ell}$. 
                Let $M_{(\ell-k,\ell)}^i,\ldots,M_{(\ell-1,\ell)}^i$ be the perfect matchings induced by $\mathcal{M}_{\ell}$ in $G_{\chi_i(\ell-k,\ell)}[V_{\ell-k}^i, V_{\ell}^i],\ldots,G_{\chi_i(\ell-1,\ell)}[V_{\ell-1}^i, V_{\ell}^i]$, respectively. 
                Let $T_{\ell+1}:=\bigcup_{a=\ell-k}^{\ell-1}\left(\bigcup_{j=\ell-k}^{a} M_{(j,a+1)}^i\right)$.\;
                }
                }
                Consider the union $\bigcup_{\ell=2}^{r}\left(\bigcup_{j=1}^{\ell-1} M_{(j,\ell)}^i\right)$ of these perfect matchings, which consists of $n_i$ vertex-disjoint $\chi_i$-coloured $k$-paths of order $r$\;
                Choose one of these $k$-paths uniformly at random and label it $P_i$\;
                Let $V_j^{i+1} := V_j^i\setminus V(P_i)$ for all $j \in [r]$ and let $n_{i+1} := n_i - 1$\;
            }
            \label{A1}
\end{algorithm}

Let $\gamma := \frac{\varepsilon \alpha}{3r}$ and $t := \left\lceil \frac{s}{\gamma n} \right\rceil$.
For $q \in[t]$, let $i_{q} := (q - 1)\gamma n + 1$.
Finally, define $\delta := 1 - \left(\frac{2 + \alpha}{2 + 2\alpha}\right)^{1/t}$, 
noting that, for any $q\in[t]$, we have \begin{equation}\label{eq:alpha} (1 - \delta)^q\left(1-\frac{1}{2k} + \alpha\right)\ge (1 - \delta)^t\left(1-\frac{1}{2k} + \alpha\right) > 1-\frac{1}{2k} + \frac{\alpha}{2}.\end{equation}
For $q\in[t]$, let $A_q$ be the bad event that $\delta(\bG[V_{\ell}^{i_{q}}, V_{j}^{i_{q}}]) < (1 - \delta)^{q}\left(1-\frac{1}{2k} + \alpha\right)n_{i_{q}}$ for some $\ell\in[r-1]$ and $\ell+1\le j\le \min\{\ell+k,r\}$.
Define $T$ to be the minimum $q\in[t]$ such that $A_q$ holds, where we let $T=t+1$ if $A_q$ does not hold for any $q\in[t]$.
That is, $T$ is the first time the bad event occurs at the start of a block. 

We begin by proving that if the bad event does not hold for $q$, then the algorithm will not abort during the $q$-th block.
More formally, for any $q\in[t]$, if $T>q$, then the algorithm does not abort in steps $i_q,\ldots,\min\{i_{q+1}-1,s\}$.
Indeed, for any $i_q\le i\le \min\{i_{q+1}-1,s\}$, we have
$$\delta(\bG[V_{\ell}^i,V_j^i]) \ge \delta(\bG[V_{\ell}^{i_q},V_j^{i_q}])-(i-i_q) > \left(1-\frac{1}{2k} + \frac{\alpha}{2}\right)n_{i_{q}}-\gamma n>\left(1-\frac{1}{2k}\right)n_{i_{q}}\ge \left(1-\frac{1}{2k}\right)n_{i},$$
for all $\ell\in[r-1]$ and $\ell+1\le j\le \min\{\ell+k,r\}$. 
Here the first inequality follows since every step in the algorithm decreases the minimum degree between any two given sets in the partition by at most $1$,
the second inequality follows from (\ref{eq:alpha}) and the fact that $i_q\le i< i_{q+1}$, and the third follows since $n_{i_q} > \frac{2\gamma n}{\alpha}$.
Hence, for any $q\in[t]$, if $T>q$, then the algorithm does not abort in steps $i_q,\ldots,\min\{i_{q+1}-1,s\}$. 
In particular, if $T>t$, then the algorithm never aborts,
thus producing the desired collection of $k$-paths. Therefore, it suffices to show that, with high probability, $T>t$. 

First, note that $T>1$ by hypopthesis.
Now, fix $q\in[t-1]$. We will find an upper bound for $\Pr(T=q+1|T>q)$. 
To this end, for each $j\in[r]$, we define $Y_j^q:=\left(\bigcup_{a=i_q}^{i_{q+1}-1} V(P_{a})\right) \cap V_j$, 
which is the set of vertices in $V_j$ that are covered by the paths $P_a$ for $i_q\le a\le i_{q+1}-1$.
Given $T>q$, we have that $Y_j^q$ is a uniformly random subset of $V_j^{i_q}$ of size $\gamma n$, for all $j\in[r]$.
Note however that $Y_j^q$ and $Y_{j'}^q$ are dependent for $j\ne j'$. 

Now, fix $j\in[r]$. For each colour\footnote{For each vertex $u\in V_j^{i_q}$ and distance $g\in\{-k,\ldots,-1,1,\ldots,k\}$, we only need to consider the specific colours we are looking for at this point in our colour pattern.
However, for the sake of brevity, we consider any colour in $[kn]$.} $c\in [kn]$, vertex $u\in V_j^{i_q}$, 
and distance $g\in\{-k,\ldots,-1,1,\ldots,k\}$ such that $j+g\in [r]$, 
consider the random variable $X_c^g(u):=|N_{G_c}(u,V_{j+g}^{i_q})\setminus Y_{j+g}^q|$, 
which is the number of neighbours of $u$ in $G_c$ contained in the set $V_{j+g}^{i_q}$ that are not covered by any path $P_a$ for $i_q\le a\le i_{q+1}-1$. 
Note that $X_c^g(u)$ is hypergeometric with $n_{i_q}-\gamma n$ draws from a population of size $n_{i_q}$, 
which contains $d_{G_c}(u,V_{j+g}^{i_q})$ success states. 
Therefore, since $A_q$ does not hold given $T>q$,
\begin{align*}
    \Ex[X_c^g(u)] &=\left(\frac{n_{i_q}-\gamma n}{n_{i_q}}\right)d_{G_c}(u,V_{j+g}^{i_q})
    \ge (1-\delta)^q \left(1-\frac{1}{2k}+\alpha\right)(n_{i_{q}}-\gamma n).
\end{align*}
By applying \thref{lem:chernoff}, when $n$ is sufficiently large we have
$$\Pr\left(X_c^g(u)<(1-\delta)^{q+1}\left(1-\frac{1}{2k}+\alpha\right)n_{i_{q+1}}\right) \le \Pr\left(X_c^g(u) \le (1- \delta)\Ex[X_c^g(u)]\right) \le e^{-\sqrt{n}},$$
as $n_{i_q} - \gamma n =\Omega(n)$.
By taking a union bound, we see that $\Pr(T=q+1|T>q) \le e^{-n/3}$ for each $q \in [t-1]$.
Hence, with high probability, 
$T>t$ and thus the algorithm succeeds in producing the required collection of $k$-paths.
\end{proof}

\section{Conclusion}
\label{Sec:Future}

Recall that the main goal of this paper was to prove \thref{thm:main Gen}, which we restate below for convenience. 

\MainTheoremGen*

The question remains as to whether the minimum degree condition in \thref{thm:main Gen} is asymptotically best possible.
As shown in~\cite{GHMPS}, the transversal threshold\footnote{The \emph{transversal threshold} for a given graph $H$ is the optimal condition which, when satisfied by every graph in a collection $\bG$, guarantees the existence of a transversal copy of $H$ in $\bG$.} for $C_n^k$ is $(1-\frac{1}{k+1})n$, as opposed to \thref{thm:main Gen},
which requires $\delta(\bG)\ge (1-\frac{1}{2k}+\alpha)n$.

\begin{ques}
\thlabel{Q1}
Let $k\ge 2$. Is there a collection of graphs $\bG=\{G_1,...,G_{m}\}$ on a common vertex set $V$ with $|V|=n$ such that $\delta(\bG)=\left(1-\frac{1}{2k}\right)n$, 
and an edge-colouring $\chi:E(C_n^k)\rightarrow [m]$, such that $\bG$ does not contain a $\chi$-coloured Hamilton $k$-cycle?
\end{ques}

It is worth noting that our proof of \thref{thm:main Gen} relied on the minium degree condition being $\left(1-\frac{1}{2k}+\alpha\right)n$ at multiple points, 
from the generation of the $k$-path collection, to the linking of our absorbing gadgets into the required absorbing structure. 
As a result, if the optimal minimum degree condition is lower than $(1-\frac{1}{2k}+\alpha)n$, 
then very different proof techniques would need to be employed.

We can show that $\delta(\bG)\ge \left(1+ \frac{1}{k+1}\right)n$ is not sufficient to ensure the presence of every colour pattern of a Hamilton $k$-cycle within $\bG$.

\begin{prop}
\thlabel{tight?}
    Let $k\ge 1$, and let $n=p(k+1)$ for some integer $p\ge 3$. 
    Then there exists a collection $\bG=\{G_1,G_2\}$ of graphs on a common vertex set $V$ with $|V|=n$, and an edge-colouring $\chi:E(C_n^k)\rightarrow [2]$, 
    such that $\bG$ does not contain a $\chi$-coloured Hamilton $k$-cycle. 
\end{prop}

\begin{proof}
Partition $V$ into sets $A_1,...,A_{k+1}$, each of size $p$. 
Let $G_1$ be the complete $(k+1)$-multipartite graph with partite sets $A_1,...,A_{k+1}$. 
Let $G_2$ be the graph constructed from $G_1$ as follows: For $i\in\left[\left\lfloor\frac{k+1}{2}\right\rfloor\right]$, replace the bipartite graph induced between $A_{2i}$ and $A_{2i-1}$ with a perfect matching. 
Finally, for $i\in\left[2\left\lfloor\frac{k+1}{2}\right\rfloor\right]$, add all edges within $A_i$.
Let $\chi$ be the edge-colouring of $C_n^k$ consisting of a single $k$-connector, from $k$ vertices to a single vertex, in colour $2$ on the vertices $v_1,...,v_{2k+1}$, and all other edges in colour $1$. 
See Figure~\ref{fig:two thirds} for an illustration when $k=2$. 

\begin{figure}[h]
\begin{center}
\begin{tikzpicture}

\begin{scope}[shift={(-6,0)}]

\node [circle, draw=black,fill=white, inner sep=0pt, minimum size=10mm] (A) at (0,2)[label= above: {$A$}]{};
\node [circle, draw=black,fill=white, inner sep=0pt, minimum size=10mm] (B) at (-1.25,0)[label=above : {$B$}]{};
\node [circle, draw=black,fill=white, inner sep=0pt, minimum size=10mm] (C) at (1.25,0)[label=above : {$C$}]{};

\draw [line width=3pt, color=red] (A)-- (B);
\draw [line width=3pt, color=red] (A)-- (C);
\draw [line width=3pt, color=red] (B)-- (C);

\end{scope}

\begin{scope}[shift={(-2,0)}]

\node [circle, draw=black,fill=white, inner sep=0pt, minimum size=10mm] (A) at (0,2)[label= above: {$A$}]{};
\node [circle, draw=black,fill=cyan, inner sep=0pt, minimum size=10mm] (B) at (-1.25,0)[label=above: {$B$}]{};
\node [circle, draw=black,fill=cyan, inner sep=0pt, minimum size=10mm] (C) at (1.25,0)[label=above: {$C$}]{};

\draw [line width=3pt, color=cyan] (A)--(B);
\draw [line width=3pt, color=cyan] (A)--(C);
\draw [line width=3pt, color=cyan, dotted] (B)--(C);

\end{scope}

\begin{scope}[shift={(4.75,0)}]
\node [std] (n-1) at (-4,1.5) [label= above: {$v_{n-1}$}]{};
\node [std] (n) at (-3,0) [label= below: {$v_n$}]{};
\node [std] (1) at (-2,1.5) [label= above: {$v_1$}]{};
\node [std] (2) at (-1,0) [label= below: {$v_2$}]{};
\node [std] (3) at (0,1.5) [label= above: {$v_3$}]{};
\node [std] (4) at (1,0) [label= below: {$v_4$}]{};
\node [std] (5) at (2,1.5) [label= above: {$v_5$}]{};
\node [std] (6) at (3,0) [label= below: {$v_6$}]{};

\draw [line width=1.5pt, color=cyan] (1)--(2);
\draw [line width=1.5pt, color=cyan] (1)--(3);
\draw [line width=1.5pt, color=cyan] (2)--(3);
\draw [line width=1.5pt, color=cyan] (2)--(4);
\draw [line width=1.5pt, color=cyan] (3)--(4);
\draw [line width=1.5pt, color=cyan] (3)--(5);

\draw [line width=1.5pt, color=red] (n-1)--(n);
\draw [line width=1.5pt, color=red] (n-1)--(1);
\draw [line width=1.5pt, color=red] (n)--(1);
\draw [line width=1.5pt, color=red] (n)--(2);
\draw [line width=1.5pt, color=red] (4)--(5);
\draw [line width=1.5pt, color=red] (4)--(6);
\draw [line width=1.5pt, color=red] (5)--(6);

\draw [line width=1.5pt, color=red] (n-1)--(-4.5,1.5);
\draw [line width=1.5pt, color=red] (n)--(-3.5,0);
\draw [line width=1.5pt, color=red] (n-1)--(-4.25,1.2);

\draw [line width=1.5pt, color=red] (5)--(2.5,1.5);
\draw [line width=1.5pt, color=red] (6)--(3.5,0);
\draw [line width=1.5pt, color=red] (6)--(3.25,0.3);

\end{scope}
\end{tikzpicture}
\caption{Graphs $G_1$ (red) and $G_2$ (blue), each with minimum degree $\frac{2}{3}n=2p$, 
and a colour pattern $\chi$ such that $\bG=\{G_1,G_2\}$ does not contain a $\chi$-coloured Hamilton $2$-cycle. 
In $\bG$, a solid colour partite set represents a copy of $K_p$, a solid edge represents a copy of $K_{p,p}$, and a dotted edge represents a perfect matching.}\label{fig:two thirds}
\end{center}
\end{figure}
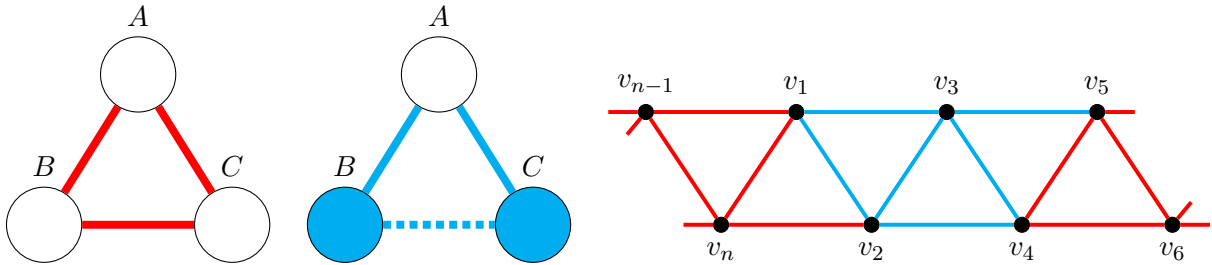

We claim that $\bG$ has no $\chi$-coloured Hamilton $k$-cycle.
For ease of notation, define $A'_{2i}:=A_{2i-1}$ and $A'_{2i-1}:=A_{2i}$ for each $i\in\left[\left\lfloor\frac{k+1}{2}\right\rfloor\right]$.
Note that, since $n\equiv 0 \pmod{k+1}$ and any copy of $K_{k+1}$ in $G_1$ uses exactly one vertex from each set, any $\chi$-coloured copy of $C_n^k$ in $\bG$ must have $v_1$ and $v_{k+2}$ in the same set. 
Moreover, $v_1,...,v_{k+1}$ must all be in different sets, and similarly $v_1$, $v_n$, and $v_{k+3},...,v_{2k+1}$ must all be in different sets. 
By symmetry, there are two cases to consider: either $v_1$ is in a set $A_i$ that induces a clique in $G_2$, or $v_1$ is in a set $A_i$ that forms an independent set in $G_2$. 

Suppose first that $v_1$, and thus $v_{k+2}$, is $A_1$, which induces a clique in $G_2$ regardless of the parity of $k$. 
Then, since $v_1,...,v_{k+1}$ are all in different sets, there is some $2\le i\le k+1$ such that $v_i\in A'_1$. 
But then $v_i$ must be adjacent to two distinct vertices (namely $v_1$ and $v_{k+2}$) in $A_1$ via edges in $G_2$, which is impossible. 
Therefore, $v_1$ must be in a set $A_i$ that forms an independent set in $G_2$.
In particular, this implies that $\bG$ does not have a $\chi$-coloured Hamilton $k$-cycle when $k\equiv 1\pmod{2}$. 

So assume $k\equiv 0\pmod{2}$, and that $v_1$, and thus $v_{k+2}$, is in $A_{k+1}$, which forms an independent set in $G_2$. 
Then, since $v_1,...,v_{k+1}$ are all in different sets, $v_2,...,v_{k+1}$ must each be in a different set $A_1,...,A_{k}$. 
For $2\le i\le k+1$, define $j_i\in[k]$ to be the index such that $v_i\in A_{j_i}$.
Note that, since all edges in $\chi$ amongst the vertices $v_2,...,v_{k+1}$ are in $G_2$, each vertex $v_i$ with $2\le i\le k+1$ is already adjacent to a vertex in $A'_{j_i}$.
Now, since $v_{k+3}$ is adjacent to each of $v_3,...,v_{k+1}$ via edges in $G_2$, and to $v_{k+2}$ via an edge in $G_1$, it follows that $v_{k+3}$ must be in $A'_{j_2}$. 
Indeed,  $v_{k+3}$ clearly cannot be in $A_{k+1}$, due to its adjacency to $v_{k+2}$ via an edge in $G_1$. 
Moreover, $v_{k+3}$ cannot be in $A'_{j_i}$ for any $3\le i\le k+1$, otherwise $v_i$ would be adjacent to two distinct vertices in $A'_{j_i}$ via edges in $G_2$, which is impossible.
Similarly, since $v_{k+4}$ is adjacent to each of $v_4,...,v_{k+1}$ via edges in $G_2$, and to $v_{k+2}$ and $v_{k+3}$ via edges in $G_1$, it follows that $v_{k+4}$ must be in $A'_{j_3}$.
Continuing in this manner, we conclude that $v_{2k+1}\in A'_{j_k}$,  and thus, since $v_1$, $v_n$, and $v_{k+3},...,v_{2k+1}$ must all be in different sets, it follows that $v_n\in A'_{j_{k+1}}$. 
However, since $A_{j_{k+1}}$ and $A'_{j_{k+1}}$ are distinct sets, there exists some $2\le i\le k$ such that $v_i$ is also in $A'_{j_{k+1}}$, 
which is impossible since $v_n$ is adjacent to each of $v_2,...,v_k$ via edges in $G_1$. Therefore, $\bG$ has no $\chi$-coloured Hamilton $k$-cycle, as required.
\end{proof}

A positive answer to \thref{Q1} would be particularly interesting as the result of Montgomery, M\"{u}yesser and Pehova~\cite{MMP} shows that the pattern threshold\footnote{The \emph{pattern threshold} for a given graph $H$ is the optimal condition which, when satisfied by every graph in a collection $\bG$, guarantees the existence of every colour pattern of $H$ in $\bG$.} for the $K_k$ factor agrees with the classical extremal threshold -- see Hajnal and Szemeredi's~\cite{HS} seminal theorem. Indeed, a Hamilton $(k-1)$-cycle contains a $K_k$-factor (when $k$ divides $n$).

Even if the answer to \thref{Q1} is negative, it is still possible the optimal correct minimum degree threshold lies between the transversal threshold and our bound in \thref{thm:main Gen}. 
Indeed, in the brief history of transversal results, we have already seen some surprising behaviour. For example, Aharoni, DeVos, Hermosillo de la Maza, Montejano, and \v{S}\'{a}mal.~\cite{ADHMS} proved that the transversal threshold for the triangle is $(\frac{26-2\sqrt{7}}{81})n^2 \approx 0.2557n^2$, in stark contrast to $n^2/4$, the edge density given by Mantel's theorem~\cite{mantel} for the appearance of a triangle. Additionally, it has been shown~\cite{MMP} that the transversal threshold and pattern threshold of large balanced complete bipartite graphs differ.
Furthermore, Anastos and Chakraborti~\cite{AC} showed that if $(G_1, \ldots, G_n)$ is an $n$-tuple of independent graphs, 
where $G_i \sim G(n, p)$ with $p \gg \log n/n^2$, then with high probability (w.h.p.) the tuple $(G_1, \ldots, G_n)$ contains a transversal Hamilton cycle but w.h.p.
each $G_i$ is not Hamiltonian. 

It is unclear the extent to which restricting to a certain number of colours or patterns can influence conditions forcing transversal structures (see, for example, {\cite[Question~3.1]{CMM}}). For certain restrictive colour patterns, our proof can be adapted to go through with the transversal threshold $\left(1+ \frac{1}{k+1}\right)$ in place of $\left(1+ \frac{1}{2k}\right)$. However, it seems that very new ideas will be needed to obtain this improvement for \emph{any} colour pattern.

\vspace{10mm}

\label{Refs}

\end{document}